 \documentclass[11pt]{amsart}
   \usepackage{amsthm,amsmath,amssymb,amscd,graphicx,enumerate, stmaryrd,xspace,verbatim, epic, eepic,color,url}

\usepackage{eurosym}
\usepackage[all]{xypic}
\SelectTips{cm}{}

   \newtheorem{thm}{Theorem}[subsection]
      \newtheorem*{thm*}{Theorem}
    
    \newtheorem{fact}[thm]{Fact}

   \newtheorem{prop}[thm] {Proposition}     
   \newtheorem{lemma} [thm]{Lemma}

   \newtheorem{cor} [thm]{Corollary}

   \newtheorem*{conjecture*}{Conjecture}

{\theoremstyle{definition}
 
          \newtheorem*{exercise*}{Exercise}
     \newtheorem{example}[thm]{Example}
 \newtheorem{definition}[thm]{Definition} 
  \newtheorem{defi}[thm] {Definition}
  \newtheorem{conv} [thm]{Convention}
 
  \newtheorem{remark} [thm]{Remark}

\newcommand{\N}{{\mathbb{N}}}

\newcommand{\PP}{{\mathbb{P}}}

\newcommand{\Z}{{\mathbb{Z}}}

 \newcommand{\cA}{{\mathcal A}}

\newcommand{\cC}{{\mathcal C}}

\renewcommand{\cL}{{\mathcal L}}

\newcommand{\cO}{{\mathcal O}}
\newcommand{\cP}{{\mathcal P}}

\newcommand{\cQ}{{\mathcal Q}}

\newcommand{\cX}{{\mathcal X}}

 \newcommand{\md}{{\underline{d}}}
  \newcommand{\mc}{{\underline{c}}}
  \newcommand{\mt}{{\underline{t}}}

 \newcommand{\mdeg}{{\underline{\deg}}}

\def\<{\langle}
\def\>{\rangle}

\newcommand{\ocO}{\overline{{\mathcal O}}}

\newcommand{\oO}{{\overline{O}}}

\newcommand{\Pic}{{\operatorname{{Pic}}}}

\newcommand{\Div}{{\operatorname{Div}}}

\newcommand{\Aut}{{\operatorname{Aut}}}

\newcommand{\double}{\genfrac..{0pt}1
{\raise -2pt\hbox{$\scriptstyle\longrightarrow$}}{\raise 4pt\hbox
{$\scriptstyle\longrightarrow$}}} 
 
\newcommand{\ssm}{\smallsetminus}

 \newcommand{\la}{\longrightarrow}
\newcommand{\ha}{\hookrightarrow}

\newcommand{\ov}{\overline}
\newcommand{\un}{\underline}

\def\Mgb{\overline{M}_g}
 
 \newcommand{\Mgnb}{\ov{M}_{g,n}}

\newcommand{\PYb}{\overline{P}^{g-1+b}_Y}
\newcommand{\PX}{\overline{P}^{g-1}_X}
\newcommand{\PXd}{\overline{P}^{d}_X}

\newcommand{\Pdgb}{\overline{P}^{d}_g}
\newcommand{\PXb}{\overline{P}^{g-1+b}_X}
\newcommand{\Pgb}{\overline{P}^{g-1}_g}

\newcommand{\Pgg}{\overline{P}^{g}_g}

\newcommand{\PXg}{\overline{P}^{g}_X}

\newcommand{\Pbb}{\overline{P}^{g-1+b}_g}

\newcommand{\Mgtb}{{\ov{M_{g}}^{\rm trop}}}

 \newcommand{\hX}{\hat{X}}
\newcommand{\hY}{\hat{Y}}
\newcommand{\hL}{\hat{L}}
\newcommand{\hM}{\hat{M}}
\newcommand{\hG}{\hat{G}}
\newcommand{\hH}{\hat{H}}
\newcommand{\hgamma}{\hat{\gamma}}

 \newcommand{\qOg}{[\Og]}  
 \newcommand{\Og}{\overline{\mathcal{OP}}_{g}^b}

   \newcommand{\oOG} {\overline{\mathcal{OP}}}  
   
      \newcommand{\Ag}{{\mathcal{A}}_{g}^b}      
    
 \newcommand{\OG}{\overline{\mathcal{OP}}^b_G}    
  \newcommand{\OGt}{\overline{\mathcal{OP}}^0_G} 
  \newcommand{\OHt}{\overline{\mathcal{OP}}^0_H} 
    \newcommand{\OGb}{\overline{\mathcal{OP}}^1_G} 
        \newcommand{\OHb}{\overline{\mathcal{OP}}^1_H} 
  \newcommand{\oOH}{\overline{\mathcal{OP}}^b_H}

\newcommand{\OP}{{\mathcal{OP}}}
\newcommand{\oOP}{{\overline{\mathcal{OP}}}}
\newcommand{\OPt}{{\mathcal{OP}}^0} 
 \newcommand{\OPb}{{\mathcal{OP}}^1}

    \newcommand{\OPGb}{{\mathcal{OP}}^b_G} 
    \newcommand{\oOPGb}{{\ov{\mathcal{OP}}}^b_G} 
     
    \newcommand{\oOPHb}{{\ov{\mathcal{OP}}}^b_H}

    \newcommand{\qOPGb}{[{\ov{{\mathcal{OP}}}}^b_G]}

       \newcommand{\qOgb}{[{\ov{{\mathcal{OP}}}}^b_g]} 
     
   \newcommand{\XN}{X^{\nu}}    
 \newcommand{\Sg}{{\mathcal {SG}}_g}

     \newcommand{\AGr}{\cA^1_G} 
      
   \newcommand{\AGb}{\cA^b_G} 
      \newcommand{\Ab}{\cA^b} 
   \newcommand{\AGt}{\cA^0_G} 
        \newcommand{\At}{\cA^0}

\begin{document}

\bibliographystyle{plain}

\title{Combinatorics of compactified universal Jacobians}

\author[]{Lucia Caporaso and Karl Christ}

 \address[Caporaso]{Dipartimento di Matematica e Fisica\\
 Universit\`{a} Roma Tre\\Largo San Leonardo Murialdo \\I-00146 Roma\\  Italy }\email{caporaso@mat.uniroma3.it}
 \address[Christ]{Dipartimento di Matematica e Fisica\\
 Universit\`{a} Roma Tre\\Largo San Leonardo Murialdo  I-00146 Roma Italy \\ and
Department of Mathematics\\
 Ben-Gurion University of the Negev\\
  P.O.Box 653, Be'er Sheva, 84105, Israel  }\email{kchrist@mat.uniroma3.it, christk@post.bgu.ac.il}

 \begin{abstract} 
We show that   the  combinatorial structure of the compactified universal  Jacobians over $\Mgb$  in degrees $g-1$ and $g$    is  governed by 
 orientations on stable graphs.
In particular, for a  stable curve   we exhibit   graded stratifications of  the compactified Jacobians 
in terms  of totally cyclic, respectively rooted, orientations on  its   dual graph.
We prove functoriality under edge-contraction of the posets of totally cyclic and rooted orientations on stable graphs.
   \end{abstract}

\maketitle


 \section{Introduction and Preliminaries}
  \subsection{Introduction}
  \label{intro}

The boundary of  the compactification of various moduli spaces   exhibits a  stratification in terms of increasingly degenerate objects. A basic example of this  phenomenon
is $\Mgb$, the compactification of the moduli space of smooth curves of genus $g\geq 2$ by stable curves, 
where the boundary strata parametrize curves with an increasing  number of nodes.

This widespread behaviour has received new attention lately thanks to recent progress in tropical   and non-Archimedean geometry. In fact, a thorough study of the boundary of $\Mgb$ and of its combinatorial incarnation
has led to a remarkable discovery: in loose words,    the Berkovich skeleton of $\Mgb$ 
(the tropicalization of $\Mgb$) is the moduli space  for the skeleta of stable curves over complete valued fields
(the moduli space of tropical curves, $\Mgtb$). 
An analogous  result   holds for other moduli spaces,
like $\Mgnb$  or the space of admissible covers. 
These facts are proved,  building upon results of 
 \cite{berkovich},   \cite{BMV} and \cite{thuillier},  in  \cite{ACP} for $\Mgnb$  and in \cite{CMR} for admissible covers; see also \cite{BPR1}, \cite{tyomkin}, \cite{viviani},
\cite{BR15} for related progress. We here investigate the compactification of the universal Jacobian.
 
As we said, the starting point has been the study of the boundary from the combinatorial point of view.
First, one shows it admits a so-called {\it graded stratification}  by a poset $\cP$, then  
 one identifies  $\cP$  with a combinatorial object interesting on its own.
For example, for $\Mgb$ the stratifying poset  is  $\Sg$, the set of all stable graphs of genus $g$ 
partially  ordered with respect to edge-contraction.  We have a     ``stratification" map, 
$\Mgb\to \Sg$,  mapping a curve   to its dual graph.  This stratification implies, roughly speaking,  
that to a degeneration of curves  there corresponds a ``dual" edge-contraction of dual graphs,  and 
to    edge-contractions there correspond degenerations of curves.

In this paper we shall extend this correspondence to degenerations of curves together with a line bundle
by suitably enriching the combinatorial counterpart. 
Moreover, we   shall prove this holds on the compactification of the universal  degree-$d$ Jacobian 
(or  degree-$d$ Picard variety)  over $\Mgb$, for $d=g-1,g$. Let us   be more precise.

Recall that for any $d\in \Z$ the compactification of the  universal degree-$d$
Jacobian is a projective morphism, 
$ 
\psi:\Pdgb\to \Mgb
$, 
whose fiber over an automorphism-free  curve,  $X$, is $\Pic^d(X)$ if $X$ is smooth,
and a compactified degree-$d$ Jacobian, $\PXd$, if $X$ is singular; we use the notation and moduli description  of \cite{Cthesis}.
   As $d$ varies  so does  $\Pdgb$,  but it is well known that there are only finitely many non-isomorphic types,
 each of which can be realized by a value of $d$ such that
$0\leq d\leq g$.  

We concentrate on the cases $d=g-1$ and $d= g$,
which are of special interest. The case $d=g-1$ has been studied extensively    because of its   connection with Prym varieties,  the Theta divisor and the Torelli problem; see \cite{be77}, \cite{Al04},  \cite{CV2}. The case $d=g$   is  notable because
  $\Pgg$ is the coarse moduli scheme of a Deligne-Mumford stack,
and its fiber over the curve $X$ is a compactified Jacobian  of N\'eron type, i.e. it compactifies the N\'eron model of the Jacobian of a regular one-parameter smoothing of $X$. 

Before studying the full space $\Pdgb$ we    study its fiber, $\PXd$, over the curve $X$.
The space $\PXd$ parametrizes
  line bundles on    partial normalizations of $X$ 
having a special multidegree; as multidegrees on $X$ coincide with divisors on the dual graph, $G$, of $X$, we call
such special multidegrees  {\it stable  divisors}.
This leads to a stratification of $\PXd$  given by the sets of nodes that are normalized, and by the sets of
stable divisors on the partial normalization. 
For a fixed curve $X$ the existence of  such a stratification was essentially   known, but 
a  combinatorially interesting incarnation for it was not, with the exception of the case $d=g-1$.
Indeed,  it was known that a divisor of degree $g-1$  is stable   if and only if it is the divisor  associated to a  {\it totally cyclic} orientation  on $G$. Preceeding the notion of stable divisor, this observation was made in \cite[Lemma 2.1]{be77}
while studying Prym varieties. 
 Independently, using the basic inequality of \cite{Cthesis}, this is a consequence of a theorem in graph theory  known as Hakimi's Theorem (originally in \cite{Ha65},   see also \cite[Theorem 4.8]{ABKS}). The graded stratification of $\PX$ by totally cyclic orientations was established   in \cite{CV2} to study the Torelli map of stable curves.  

We will prove    results of a similar type in   case $d=g$, and show  that
 $\PXg$ has a graded stratification by the poset of {\it rooted} (generalized) orientations 
on $G$; see Definition~\ref{1ordef}.  In particular,  we   show that a divisor   is stable if and only if it is the divisor associated to a rooted  orientation. We note that from this and \cite[Lemma 3.3]{ABKS} it easily follows that the notions of break divisor, as introduced in \cite{MZ08}, and of stable divisor   coincide.

We  will introduce for a stable graph $G$ two graded posets: the poset $\OGt$ of totally cyclic  orientation  classes on  spanning   subgraphs of $G$,  and the poset $\OGb$ of rooted orientation classes on  spanning  subgraphs of $G$. 
  We treat the cases $d=g-1$ and $d=g$ simultaneously, so we write $b=0,1$ and $d=g-1+b$.
 By mapping a point to its stratum  we   get a graded stratification map  $\PXb \rightarrow  \OG$; see Theorem~\ref{PX}.

Then we extend our analysis over $\Mgb$ which, as we said,     is stratified by $\Sg$
  ordered by edge-contraction.
The   goal is to endow   $\Pgb$ and $\Pgg$  with a graded stratification   compatible with the one of $\Mgb$.
 In order to do that 
 we need to study the     behaviour of the posets $\OG$ under edge-contractions. 
 This is a   combinatorial problem which, to our knowledge, has never been studied.
 Our   main result    here is  Theorem~\ref{fthm},    stated informally   as follows.
\begin{thm}
\label{fthmintro}
Let $G$ and $G'$ be stable graphs and 
let $\gamma:G\to G'$ be a non-trivial edge-contraction. Then, for  $b=0,1$, 
we have a natural   quotient  of posets $  
{\overline {\gamma}}_*:\oOPGb\to \oOP^b_{G'} . 
$ 
   \end{thm}
 
Taking the   action
  of $\Aut(G)$ into account we
  have 
  a quotient of posets,  $\qOg \to \Sg$, whose fiber over $G$ is $\OG/\Aut(G)$.
The theorem below, describing our    compactified Jacobians  in terms 
of orientations on graphs,
  summarizes the algebro-geometric results.

\begin{thm}
\label{mainintro}
Let $b=0,1$.
 The following diagram   is commutative. The four horizontal maps, denoted by $\sigma$, are  graded stratification maps,
 and the vertical map $\mu$ is a quotient of posets.
  $$\xymatrix@=.5pc{
\Pbb  \ar@{->>} [rrrrrrrrr]^{\sigma}    \ar[ddddd]_{\psi}&&&&&&&&&\qOg\ar [ddddd]^{\mu}&&\\
  &&&&&&&&&   && \\
  &\PXb/\Aut(X) \ar[dd] \ar@{_{(}->}[luu]    \ar@{->>}[rrrrrr]^{\sigma}  &&&&&&\qOPGb \ar [dd]\ar@{^{(}->}[rruu]  &&&   &&&&  \\
  &&\PXb\ar[lu]\ar[ld]  \ar@{->>}[rrr]^{\sigma} &&&\OG\ar[rru]\ar[rrd] &&&& &&\\
 &[X]   \ar@{|->}[rrrrrr]  \ar@{|->}[ld] &&&&&&[G]  \ar@{|->}[rrd] &&& &&\\  
 \Mgb \ar@{->>}  [rrrrrrrrr]^{\sigma} &&&&&&&&&\Sg&&\\
}$$
\end{thm}
 
 This   holds in degrees $g-1$ and $g$ and can be easily extended to degree $g - 2$  (by taking the residual of the degree $g$ case).

The theorem gives the sought-for combinatorial presentation of the compactified Jacobian of a curve, and of the compactified universal Jacobian over $\Mgb$. The next question now is  to provide the  tropical/non-Archimedian version of the theorem, 
starting from the fact that   the left-bottom corner of the diagram should be occupied by the moduli space of tropical curves,
$\Mgtb$, while the right side should be the same, up to isomorphism.
This will involve  constructing   skeleta of $\PXd$ and $\Pdgb$ as moduli spaces of suitable polyhedral objects. This   research direction relates to results of
 \cite{BR15}, where  the skeleton of the Jacobian of a curve over a valuation ring is shown to be the Jacobian of the skeleton of the curve. Results of \cite{Sthesis} suggest that similar methods can be used to treat the case of compactified Jacobians. Results of \cite{Kthesis} indicate that one can extend this description  to the universal setting on the combinatorial side. 
 We refer to  \cite{AP18} for an approach to this program using
  the compactification constructed in \cite{Esteves}, which   coincides with the one we use here for degree $g$, but not for degree $g - 1$.

The paper begins with some elementary combinatorial preliminaries. Then
Section~\ref{secfun}     establishes  the  main results for    orientations   and  their behavior  under edge-contractions,
proving Theorem~\ref{fthmintro}. 
Our work  here has been influenced   by \cite{Gi07} and \cite{Ba14}, which study the interplay between orientations and the divisors they define. 
 In Section~\ref{finsec}  we treat
   compactified   Jacobians   and   prove   
 Theorem \ref{mainintro}.

 \subsection{Graphs}
 \label{graphsec}
Throughout the paper 
 $G$ denotes a   vertex-weighted finite graph; we allow loops and multiple edges.  We denote by
 $V=V(G)$   the set of  vertices of $G$, by
   $E=E(G)$    the set of   edges of $G$ and by
   $ 
   w :V\to \N 
   $ 
   the weight function of $G$.
   We write $c(G)$ for the number of connected components of $G$.
   The {\it genus}, $g=g(G)$, of $G$,  is  
   $$g(G):=\sum_{v\in V} w(v) -|V|+|E|  +c(G).$$

We think of an edge of $G$ as the union of two {\it half-edges}, each of which has a vertex of $G$ as
    end, so that the ends of an edge $e$ are the ends of its half-edges and $e$ is a {\it loop} if the two ends coincide.
   We write $H=H(G)$ for the set of half-edges of $G$. We have a natural two-to-one surjection $H\to E$,
and we write $\{h^+_e,h^-_e\}$ for the preimage of $e\in E$.
The {\it degree}, $\deg v$, of a vertex $v$ is the number of   half-edges whose end is $v$.

For a non empty $Z\subset V$,  we write $Z^c:=V\ssm Z$. The {\it induced subgraph},
$G[Z]\subset G$, is the subgraph whose vertex-set is $Z$, whose edge-set is the set of all
edges of $G$ having both ends in $Z$, and whose
 weight function is the restriction to $Z$ of  the one of $G$.
 We set
$$
g(Z):=g(G[Z])=|E(G[Z])|-|Z|+c(G[Z])  +\sum_{v\in Z} w(v).
$$

   If $S\subset E$ is a set of edges of $G$, we write $G-S$ for the graph obtained from $G$ by removing $S$;
   notice that $G$ and $G-S$ have the same vertices, in other words $G-S$ is a so-called {\it spanning} subgraph of $G$.
   We denote by $\langle S\rangle$ the subgraph of $G$ spanned by $S$, so that  $E(\langle S\rangle)=S$ and the vertices of $\langle S\rangle$ are   the vertices adjacent to the edges in  $S$.

A {\it cut} of $G$ is a  set of edges, $S\subset E$, such that for a
    partition $V=Z\sqcup Z^c$, with $\emptyset\subsetneq Z \subsetneq V$, our
 $S$ is the set of all edges adjacent to both $Z$ and $Z^c$.
 We also write $S=E(Z,Z^c)$ for such a cut.
 For a   non empty cut $S$ we have
 $ 
 c(G)<c(G-S).
 $       
A   {\it bridge} is an edge  such that $\{e\}$ is a cut.
We denote by
$G_{br}\subset E$ the set of bridges   of $G$.
 
\begin{remark}
\label{cutrk}
 Let $S\subset E$ be a cut of $G$ and let $H\subset G$ be a subgraph.
 Then either $S\cap E(H)=\emptyset$ or $S\cap E(H)$ is a cut for $H$.
 \end{remark}

\begin{remark}
 For any $S\subset E$ we have
$ 
 g(G)\geq g(G-S),
 $ 
with equality   if and only if $S\subset G_{br}$.
\end{remark}

A   {\it {morphism}}  between two graphs, $\eta:G\to G'$, is given by two maps,
$\eta_V:V(G)\to V(G')$ and $\eta_E:E(G)\to E(G')\cup V(G')$ such that 
$\eta_E(e)$ has ends $\eta_V(v)$ and $\eta_V(w)$ for any $e\in E(G)$ whose ends are $v$ and $w$. 
We sometimes write just $\eta=\eta_E$ and $\eta=\eta_V$.

An   {\it {isomorphism}}  between two graphs, $\alpha:G\to G'$, is a morphism such that
$\alpha_V$ is a bijection,   $\alpha_E:E(G)\to E(G')$ is a bijection, and such that for every $v\in V(G)$ the weight of $\alpha_V(v)$ equals the weight of $v$. An isomorphism induces also a bijection between the half-edges of $G$ and $G'$.
An {\it automorphism} is an isomorphism of $G$ with itself.
We denote by $\Aut(G)$   the group of automorphisms of $G$.

 $G$ is {\it semistable} if it is connected,   $g(G)\geq 2$, and has no vertex of weight $0$ and degree less than $2$. 
$G$ is {\it stable} if it is semistable and has no vertex of weight $0$ and degree less than $3$. 
The set of all stable graphs of genus $g$ is denoted by $\Sg$.
 Notice that $\Sg$ is finite.

 \subsection{Edge-contractions}
 \label{consec}
Fix $S\subset E$.
The {\it (weighted) contraction} of $S$ is
a map of weighted graphs,  
$ 
\gamma: G\to G/S 
$ 
(introduced in \cite{BMV}).  Informally   $\gamma$ is
given by contracting to a vertex every edge in $S$, and such that the weight of a vertex  $v$
 of $G/S$ equals the genus of the subgraph of $G$ which gets contracted into $v$.
Rigorously, consider the subgraph, $\langle S\rangle \subset G$, spanned by the edges in $S$
and let  $\langle S\rangle =  H_1\sqcup  \ldots \sqcup H_m$ be its decomposition in connected components.
Now set
$$
V(G/S):= V(G)\ssm V(\langle S\rangle)\sqcup \{v_1,\ldots, v_m\}, \quad \quad E(G/S):=E(G)\ssm S.
$$
We have two   maps, 
\begin{equation}
\label{gammaVE} 
\gamma_V:V(G)\la V(G/S)  \quad \text{and} \quad \gamma_E:E(G)\la E(G/S)\cup  V(G/S),
\end{equation}
where  $\gamma_V$
is the identity on    
$V(G)\ssm V(\langle S\rangle)$
and maps   every vertex of $H_i$ to $v_i$, and  $\gamma_E$  is   the   identity on $E(G)\ssm S$ and maps every $e\in S$  to $v_i$
such that $e$ lies in $H_i$.
It is clear that $\gamma_V$ and $\gamma_E$ determine a morphism of graphs
$ 
\gamma: G\to G/S$, as   wanted.
Finally, the weight function $w_{/S}:V(G/S)\to \N$
is defined as follows: 
$$
w_{/S}(v)=g(\gamma^{-1}(v)).
$$
Indeed,  $\gamma^{-1}(v)$ is the subgraph of $G$ induced by
the subset $\gamma_V^{-1}(v) \subset V(G)$,
hence its genus is well defined.

For   convenience we   view the identity of $G$ as  the trivial contraction.

For $S\subset E$ we set 
\begin{equation}
 \label{G(S)}
G(S) :=G/(E\ssm S).
\end{equation}
\begin{remark}
\label{ftrivrk} 
 \begin{enumerate}[(a)]
 \item
   \label{ftriv0}
 $G$ is connected if and only if $G/S$ is connected.
 \item
  \label{ftrivg}
 $g(G)=g(G/S)$.
\item
 \label{ftriv1}
If $G$ is   stable, or semistable, so is $G/S$.
\end{enumerate}
\end{remark}
  \subsection{Posets}
 \label{posec}
A {\it poset}, $(\cP, \leq)$, or just $\cP$, is a set  partially ordered with respect to $`` \leq"$. Let $p_1,p_2\in \cP$.
We say that $p_2$ {\it covers} $p_1$ if $p_1< p_2$  and if there is no $p'\in \cP$
such that $p_1<p'<p_2$.

 Let  $(\cP,\leq_{_\cP})$ and $(\cQ,\leq_{_\cQ})$ be  two posets. We say that a map
$ 
\mu :\cP\to\cQ.
$ 
 is a {\it morphism of posets} if   $p_1\leq _{_\cP} p_2$
implies  $\mu(p_1)\leq _{_\cQ} \mu(p_2)$.
We say that $\mu$ is a {\it quotient  (of posets)} if  for any $q_1,q_2\in \cQ$
such that $q_1\leq _{_\cQ}q_2$ there exist  $p_1\in \mu^{-1}(q_1)$ and  $ p_2\in \mu^{-1}(q_2)$ such that $p_1\leq _{_\cP} p_2$.
In particular, a quotient is a surjective morphism of posets.

We will apply the following trivial lemma  a few times.
\begin{lemma}
\label{posettriv}
 Let $\cP$ be a finite poset and $\sim$ an equivalence relation on $\cP$.
Let $\pi:\cP \to \ov{\cP}=\cP/_\sim$ be the    quotient.
Assume the  following   holds
 
 For every  $ x,y\in \cP$ with $y\geq x$ and for every $y\sim y'$ there exists $x'\sim x$ such that $y'\geq x'$.
 
 Then 
 $\ov{\cP}$ is a poset as follows:
for $\ov{x},\ov{y}\in\ov{\cP}$  set $\ov{y}\geq \ov{x}$ if there exist $x'\sim x$ and $y'\sim y$ such that $y'\geq x'$.
Moreover $\pi$ is a quotient of posets.
\end{lemma}
The lemma   holds if we switch roles between $x$ and $y$, i.e. if we assume that
  for every $x\sim x'$ there exists $y'\sim y$ such that $y'\geq x'$.

A {\it rank} on a poset $\cP$ is a map
$ 
\rho: \cP\to \N
$ 
 such that
if $p_2$  covers $p_1$ then $\rho(p_2)=\rho(p_1)+1$. Of course, $\N$ is a poset and a rank is a morphism of posets.
 A poset endowed with a rank   is called a {\it graded poset}.

\begin{defi}
\label{stratdef}
 Let $M$ be an algebraic variety and let $\cP$ be a poset.
A {\it stratification of $M$ by $\cP$} is a  partition of $M$
$$
 M=\sqcup_{p\in \cP}M_p
$$
such that the following hold for every $p,p'\in \cP$.
\begin{enumerate}
 \item
 the {\it stratum} $M_p$ is irreducible and  quasi-projective;
 \item
if $M_p\cap \overline{M_{p'}}$ is not empty, then $M_p\subset\overline{M_{p'}}$;
\item
 $M_p\subset\overline{M_{p'}}$ if and only if $p\leq p'$.
\end{enumerate}
A stratification of $M$ by $\cP$ is called {\it graded} if   the following is a rank  on $\cP$
$$
\cP\la \N; \quad \quad p\mapsto \dim M_p.
$$
Let $\sigma: M \twoheadrightarrow \cP$ be a surjective map. We call $\sigma$ a {\it (graded) stratification map}
if the fibers of $\sigma$ form a (graded) stratification of $M$ by $\cP$.
 
\end{defi}
  \subsection{Generalized orientations}
  Let $G$ be a graph and $e$ an edge of $G$. 
   An {\it orientation} on $e$ is the assignment of a direction
  so that one half-edge of $e$ is the starting half-edge   and the other
is the  ending half-edge. Accordingly, the  vertex adjacent to the starting half-edge 
 will be called the  {\it source} of $e$, and the   vertex adjacent to the ending half-edge 
 will be called the {\it target} of $e$.
  If $e$ is a loop then its base  vertex is both source and target.

   An {\it orientation}, $O$, on  $G$ is the assignment of an orientation on every edge  of $G$.
 If $x\in V$ is the source (respectively, the target)  of $e\in E$ we say that $e$ is $O${\it-outcoming} from $x$
 (resp.  $O${\it-incoming} at $x$).

A {\it generalized orientation} on    $G$ is the assignment, for every $e\in E$, of either
an orientation on $e$, or  of  both orientations on $e$; in the latter case we say that $e$ is {\it bioriented}.
So, a   bioriented edge  has both its ends as targets and sources.

For $b\in \N$ a  {\it $b$-orientation} is a generalized orientation having exactly $b$ bioriented edges. We thus recover
usual orientations  as 0-orientations  (which we shall continue to call ``orientations"  to 
  ease the terminology)

In this paper, we shall mostly be interested in the cases $b=0,1$.

\begin{defi}
\label{1ordef}
Let $G$ be a graph.

 An orientation (i.e. a 0-orientation) on $G$ is {\it totally cyclic}   
 if it has   no   {\it directed cut}  i.e.   if   every     non empty cut  $E(Z,Z^c)$  has an edge with target in $Z$ and an edge with target in $Z^c$.
  
A  1-orientation  on $G$  with bioriented edge $e$  is {\it rooted}, or $e$-{\it rooted}, if for every  $Z\subsetneq V$
such that $e\in G[Z]$,
the cut $E(Z,Z^c)$ contains an edge with target in $Z^c$.
 \end{defi}
  
We denote 
$$
 \cO^0(G):=\{O:\   O \text{ is a totally cyclic orientation on } G\}  
$$
and  $$
 \cO^1(G):=\{O:\   O \text{ is a rooted 1-orientation on } G\}.
$$
 
The  terminology  ``totally cyclic" and ``rooted"  is    motivated by \ref{F1} \eqref{F12},
 and  \ref{lmfree}, respectively.
 
Let $G$ be a cycle.
We say that $G$ is {\it   cyclically 
oriented} if it is given a totally cyclic orientation  (of course, a cycle admits 
  exactly two totally cyclic orientations).
From  \cite[Lemma 2.4.3]{CV} we have:
\begin{fact}
\label{F1}
\begin{enumerate}[(a)]
   \item
  \label{F11}
$ \cO^0(G)$ is not empty if and only if $G$ is free from bridges.
 \item  \label{F12}
Let  $G$ be connected.  An orientation    on $G$ is totally cyclic if and only if every pair of vertices   is contained in a  cyclically 
oriented cycle.
  \end{enumerate}
\end{fact}

\begin{conv}
\label{empty}
 Assume $G$ has no edges.
The  empty  orientation will be considered  totally cyclic, so that
 $\cO^0(G)$ consists of exactly that orientation.  
 
 If $G$ consists of a single vertex,   the empty orientation will be considered rooted,  so that $\cO^0(G)= \cO^1(G)$. 
 \end{conv}
 
By definition, an orientation on a graph is totally cyclic if and only if its restriction to every connected component of $G$ is totally cyclic.

\begin{remark}
\label{0e1}
 Let $O$ be a totally cyclic orientation on a connected graph $G$. For any   $e$ of $G$, let $O_e$ be the 1-orientation  having
 $e$ as bioriented edge  and such that every remaining edge is oriented according to $O$. The  definition implies that $O_e$ is rooted.
 This gives an injection (not a surjection)
 $$
 \cO^0(G)\times E \la \cO^1(G);\quad \quad (O,e)\mapsto O_e.
 $$
\end{remark}

 \begin{lemma}
\label{LmO}
 $ \cO^1 (G)$ is not empty if and only if $G$ is connected.
\end{lemma}
\begin{proof}
If $G$ admits  a  rooted 1-orientation then, 
by definition,  every cut $E(Z,Z^c)$ is non empty, hence $G$ must be connected.

Conversely, let $G$ be connected and set
$ 
G-G_{br}=G_1\sqcup G_2 \sqcup \ldots \sqcup G_n
$ 
with $G_i$ connected for $i=1,\ldots n$.
Of course, $G_i$ is bridgeless for every $i$, hence we can fix on $G_i$   a totally cyclic orientation, $O_i$.

We pick an edge $e$ of $G_1$ and consider 
  the 1-orientation on $G_1$ having
 $e$ as bioriented edge  and such that every other edge is oriented according to $O_1$. 
This is a  rooted  1-orientation, as noted in Remark~\ref{0e1}.
We fix this orientation on $G_1$  from now on, and we fix the orientations $O_2,\ldots , O_n$ on the remaining $G_i$.

Let us show how to orient   $G_{br}$ to obtain a rooted 1-orientation.
Let $B_1\subset G_{br}$ be the set of bridges adjacent to $G_1$ and, up to reordering $G_2,\ldots, G_n$,
let $G_2,\ldots , G_{n_1}$ be   adjacent to $B_1$, so that the following subgraph of $G$
$$
H_2=G_1\cup B_1  \cup G_2\cup \ldots \cup G_{n_1}
$$
is connected. Since $G$ is connected, $n_1\geq 2$. Orient every edge in $B_1$ pointing away from $G_1$.
It is easy to check that the so obtained 1-orientation on $H_2$ is rooted. If $H_2=G$ we are done. 
If not we iterate as follows.
Let $B_2\subset G_{br}$ be the set of bridges adjacent to $H_2$ and  
let $G_{n_1+1}, \ldots G_{n_2}$  be the components not contained in $H_2$ and adjacent to $B_2$, so that the following  $$
H_3=H_2\cup   B_2 \cup G_{n_1+1}\cup \ldots \cup G_{n_2}
$$
is connected.
Orient every edge in $B_2$ away from $H_2$  so that the so-obtained 1-orientation is rooted. If $H_2=G$ we stop, otherwise we iterate.
Since $G$ is connected, after a finite number, say $m$, of iterations  we get    $H_m=G$.  \end{proof}

 \subsection{Divisors of generalized orientations}
The group of divisors on $G$, written $\Div(G)$, is the free abelian  group generated by $V$.
We shall identify
$ 
\Div(G)= \Z^V 
$ 
and denote a divisor on $G$ by
$\md=\{\md_v\}_{v\in V}$.

The  degree of a divisor $\md$ is defined as $|\md|=\sum_{v\in V}\md_v$
and we write $\Div^k(G)$ for the set of divisors of   degree $k$.

If $\md,\md'\in \Div(G)$ are such that $\md_v\leq \md'_v$ for every $v\in V$, we write
$\md\leq \md'$.

If $S\subset E$, then $G$ and $G-S$ have the same vertices, hence we shall    identify
$ 
\Div(G)=\Div(G-S).
$ 

If $Z\subset V$ we write $\md_Z$ for the restriction of $\md$ to $Z$ and $|\md_Z|=\sum _{v\in Z}\md_v$.
 
To a  generalized orientation $O\in \cO^b(G) $ (recall that if $E(G)$ is not empty $b$ is the number of bioriented edges)  we associate a
  divisor,    $\md^O\in \Div(G)$,  whose 
  $v$ coordinate, for every $v\in V$, is
  defined as follows 
$$
 \md^{O}_v:=    \begin{cases}   w(v)-1+\mt^O_v   & \text{ if } E(G)\neq \emptyset \\
 w(v)-1+b & \text{ if } E(G)= \emptyset\\
 \end{cases}
$$
where
$ 
\mt^{O}_v
$ 
denotes  the number of half-edges having $v$ as target, so that $\mt^{O}:=\{\mt^{O}_v\}$ is also in $ \Div(G)$. 
  If $G$ is connected and $O\in \cO^b(G)$
we have
\begin{equation}
 \label{degto}
|\md^O|= g(G)-1+b. 
\end{equation}

For any $Z\subset V$ we  denote by $t^O(Z)$ the number of 
 edges  not contained in $G[Z]$  having   target in $Z$, and by $b(Z)$ the number of bioriented edges contained in $G[Z]$.
 Notice the following
 
\begin{equation}
 \label{tOZ}
t^O(Z)=   \sum_{z\in Z}\mt_z^O - |E(G[Z])|   -b(Z).
  \end{equation}
 The following trivial lemma generalizes \eqref{degto}.
 
\begin{lemma}
 Let $O$ be a $b$-orientation on  $G$  and let $Z\subset V$ be such that
  $G[Z]$ is connected.
  Then
 \begin{equation}
 \label{degtZ}
|\md_Z^O|= g(Z)-1+b(Z)+t^O(Z). 
\end{equation}
\end{lemma}
\begin{proof}
We have  
 $$
  |\md^{O}_Z|=\sum_{z\in Z}\md_v^O=\sum_{z\in Z}(w(z) -1+ \mt_z^O)= \sum_{z\in Z}w(z) -|Z|+\sum_{z\in Z}\mt_z^O.
 $$
Now,
$g(Z)=\sum_{z\in Z}w(z) -|Z|+ |E(G[Z])|+1$ hence, by \eqref{tOZ},
  
  $ 
  |\md^{O}_Z|= g(Z)-1-|E(G[Z])|+\sum_{z\in Z}\mt_z^O = g(Z)-1+b(Z)+t^O(Z).
 $ 
 \end{proof}

The following  lemmas characterize totally cyclic and rooted orientations. They are slight generalizations of \cite[Lemma 1]{Be08} and the remark thereafter.
 
\begin{lemma}
\label{lm0} Let $O$ be a  $0$-orientation on a connected graph $G$.  
 The  following are equivalent.
\begin{enumerate}[(a)]
 \item
  \label{lm0a} 
 $O$ is totally cyclic.
 \item
 \label{lm0b}  $ t^O(Z)> 0$ 
 for every non empty     $Z\subsetneq V$. 

  \item
 \label{lm0c} 
  $ t^O(Z)> 0$
 for every non empty      $Z\subsetneq V$ with $G[Z]$   connected.

 \item
  \label{lm0d}
  $ 
  |\md^{O}_Z|>g(Z)-1 
 $ 
 for every non empty  $Z\subsetneq V$ with  $G[Z]$ connected.
 \end{enumerate}
\end{lemma}
\begin{proof}
 \eqref{lm0a} $\Rightarrow$ \eqref{lm0b}.   
By hypothesis the cut $E(Z,Z^c)$ must have some edge with target in $Z$, hence  $ t^O(Z)> 0$.
  
  \eqref{lm0b} $\Rightarrow$\eqref{lm0c}  is obvious.
 
 \eqref{lm0c} $\Rightarrow$\eqref{lm0d}.  
By \eqref{degtZ} (with $b(Z)=0$) and by hypothesis we have 
$$
   |\md^{O}_Z|=g(Z) -1 + t^O(Z)>g(Z) -1.
$$
 
  \eqref{lm0d} $\Rightarrow$ \eqref{lm0a}.
  Let $E(U,U^c)$ be a cut in $G$,  
we must prove that $E(U,U^c)$ is not a directed cut.    
  Let $Z\subset U$ such   that $G[Z]$ is a connected component of $G[U]$.
  Of course, $E(Z,Z^c)\subset E(U,U^c)$.
By \eqref{degtZ} we have 
  $$
 t^O(Z)= |\md^{O}_Z|-(g(Z)-1)>0
  $$
  where the inequality follows by   hypothesis. Hence $E(U,U^c)$ has an edge with target in $Z$, hence in  $U$.
The same argument applied to $U^c$ shows that
 $E(U,U^c)$ has an edge with target in $U^c$. \end{proof}

\begin{lemma}
\label{lmO1} Let $O$ be a  non empty 1-orientation on    $G$  and let $e$ be its bioriented edge. 
 The  following are equivalent.
\begin{enumerate}[(a)]
 \item
  \label{lmO1a} 
 $O$ is $e$-rooted.
 \item
 \label{lmO1b} 
  $ t^O(Z)> 0$ 
for every  non empty   $Z\subset  V$ with $e\not\in G[Z]$.
  \item
 \label{lmO1c} 
  $ t^O(Z)> 0$  for every  non empty   $Z\subset  V$ such that $G[Z]$ is connected  and  $e\not\in G[Z]$.

 \item
  \label{lmO1d}
    $ 
  |\md^{O}_Z|>g(Z)-1 
 $ 
for every $Z\subsetneq V$ such that $G[Z]$ is connected.

 \end{enumerate}
\end{lemma}
\begin{proof}
 \eqref{lmO1a} $\Rightarrow$ \eqref{lmO1b}.   
By hypothesis $e\in G[Z^c]$. As $O$ is rooted   the cut $E(Z,Z^c)$ must have some edge with target in $Z$, hence  $ t^O(Z)> 0$.
 
   \eqref{lmO1b} $\Rightarrow$\eqref{lmO1c}  is obvious.
 
 \eqref{lmO1c} $\Rightarrow$\eqref{lmO1d}.   
 If $e\not\in G[Z]$ the proof is   the same as for Lemma~\ref{lm0}.
 If $e\in G[Z]$  we apply \eqref{degtZ}; as     $b(Z)=1$ we get
 $$
  |\md^{O}_Z|=   g(Z)  + t^O(Z) \geq g(Z)>g(Z)-1. 
 $$

  \eqref{lm0d} $\Rightarrow$ \eqref{lm0a}.
  Let $E(U,U^c)$ be a cut in $G$ with $e\in G[U]$. Let $W$ be a connected component of $G[U^c]$,
  it suffices to show that $E(U,U^c)$ contains an edge with target in $W$.
Now  \eqref{degtZ} applied to $W$        yields
 $$
  g(W) -1 + t^O(W)=  |\md^{O}_W|>g(W)-1,
 $$
by hypothesis. Hence $t^O(W)>0$,  as wanted.  \end{proof}
  \subsection{Equivalence of generalized orientations}
  \begin{defi}
  \label{defequivo}
We define   two generalized orientations, $O$ and $O'$,    on a graph $G$ 
to be {\it equivalent}, and write
 $O\sim O'$, if $\md^O=\md^{O'}$.
 
 We denote by $\oO$ the equivalence class of $O$.
 \end{defi}

\begin{remark}
\label{eqOrk}
Let $O$ and $O'$ be two $b$-orientations, with $b=0,1$. By Lemmas~\ref{lm0} and \ref{lmO1}, if $O\sim O'$ then $O$ is totally cyclic (resp. rooted) if and only if so is $O'$. 
\end{remark}
We now
 introduce  the sets of equivalence classes of totally cyclic orientations, and of rooted 1-orientations, on $G$ written
\begin{equation}
 \label{ocG}
\ocO^0(G):=\cO^0(G)/\sim  \quad \quad \text{and}  \quad \quad \ocO^1(G):=\cO^1(G)/\sim.
\end{equation}

 \begin{remark}
 \label{revpath}{\it  Equivalence of 1-orientations through   reversal of   directed paths.}
 Let $O$ be a 1-orientation whose bioriented edge $e$ has ends $v_0, v_1$.
We say that a path $P\subset G$  is $O$-directed from $e$ to $v$, with   $v\neq v_0,v_1$,
if the first edge of $P$ is $e$  and if the component of $P-e$ containing $v$ is  a directed path with $v$ as target.

 Let $P\subset G$ be an $O$-directed path from $e$ to $v_{n+1}$ as in   the picture below
 
   
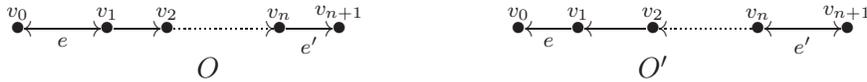
\begin{figure}[h]
\begin{equation*}
\xymatrix@=.4pc{
 *{\bullet} \ar@{<->}[rrr] ^(1){v_1}^(0){v_0}_(.5){e}&&& *{\bullet} \ar@{->} [rr]^(1){v_2} &&*{\bullet} \ar@{.>} [rrr]&&&*{\bullet}\ar@{->}[rr] ^(0){v_n}^(1){v_{n+1}}_(.5){e'}&& *{\bullet} 
 &&&&&&
  *{\bullet} \ar@{<-}[rr] ^(1){v_1}^(0){v_0}_(.5){e}&& *{\bullet} \ar@{<-} [rr]^(1){v_2} &&*{\bullet} \ar@{<.} [rrr]&&&*{\bullet}\ar@{<->}[rrr] ^(0){v_n}^(1){v_{n+1}}_(.5){e'}&&& *{\bullet}  \\
 &&&&&&O& &&&&&&&&&&&&&O'&&&&&&&}
\end{equation*}
\end{figure}
 
 Let $e'\subset P$ be the last edge of the path, so that the ends of $e'$ are $v_n$ and $v_{n+1}$.
 Define a new 1-orientation, $O'$ on $G$ as follows.
 Let $e'$ be the bioriented edge, reverse the orientation on every remaining edge of $P$,
 and fix on $e$ the orientation from $v_1$ to $v_0$. Notice that $P$ is an $O'$-directed path from $e'$ to $v_0$.
 Let $O'$ coincide with $O$ on   the remaining edges of $G$. 
 It is clear that $O$ and $O'$ are equivalent.
\end{remark}

\begin{lemma}
  \label{lmfree} 
 Let $O$ be a non empty 1-orientation on a connected graph $G$  and  let $e$ be its bioriented edge.
 The  following are equivalent.
\begin{enumerate}[(a)]
 \item
  \label{lmfreea} 
 $O$ is $e$-rooted.
 \item
 \label{lmfreeb} 
 For every $v\in V$ there exists an $O$-directed path from $e$ to $v$.
    \item
 \label{lmfreec} 
 For every $e'\in E$ there exists a 1-orientation $O'$  whose bioriented edge is $e'$ and such that
 $O\sim O'$.
 \end{enumerate} 
\end{lemma}
\begin{proof}
  \eqref{lmfreea} $\Rightarrow$ \eqref{lmfreeb}.   
  Let $x,y$ be the ends of $e$ and let $Z_1=\{x,y\}$.
  Since $O$ is $e$-rooted and $e\in G[Z_1]$  the set, $W_1$, of vertices in $Z_1^c$ that are targets of edges with source in $Z_1$ is not empty.
  Set $Z_2=Z_1\cup W_1$.
  If $W_1$ contains $v$ we are done. If not, we iterate as follows.
  As $O$ is rooted   the set, $W_2$, of vertices in $Z_2^c$ that are targets of edges with source in $Z_2$ is not empty.
By construction,  every vertex $w$ in $W_2$ is the target of an edge with source in $W_1$, and hence $w$ is the last vertex
of a directed path starting with $e$.
If $W_2$ contains $v$ we are done, otherwise we iterate.
Since $G$ is connected, after finitely many steps this process includes all vertices of $G$, so we are done.
  
   \eqref{lmfreeb} $\Rightarrow$ \eqref{lmfreec}.     
   Let $e'$ be oriented from $v$ to $w$ and let $P$ be an $O$-directed path from $e$ to $v$. We define $O'$ as the 1-orientation
   obtained by reversing the orientation of $P$, as defined in
   \ref{revpath}.
   
 \eqref{lmfreec} $\Rightarrow$ \eqref{lmfreea}.    
 By contradiction, suppose $O$ is not rooted. Hence there exists a   cut $E(Z,Z^c)$ directed away from $Z$ and such that
 $e\in G[Z^c]$. Up to replacing $Z$ with a subset, we can assume that $G[Z]$ is connected.
We thus have $t^O(Z)=0$ and, as $e\not\in G[Z]$,
  \begin{equation}
  \label{eqlmfreea}
   |\md^O_{Z}|=g(Z)-1+t^O(Z)=g(Z)-1.
\end{equation}
   Pick $e'\in G[Z]$ and let $O'$ be a 1-orientation with $e'$ as bioriented edge such that $O\sim O'$, which exists by hypothesis.
 As $e' \in G[Z]$ we have
$$
 |\md^O_{Z}|= |\md^{O'}_{Z}|=g(Z)+t^{O'}(Z)\geq g(Z)
$$
a contradiction with \eqref{eqlmfreea}.
\end{proof}

\subsection{The posets of bridgeless and connected subgraphs}
\label{brisec}
Let $G$ be a graph and $E$ its edge-set.
The set of all subsets of $E$,  written $\cP(E)$, will be considered as a poset 
with respect to reverse inclusion, i.e. we set
\begin{equation}
 \label{ri}
S\leq S' \quad \text{  if } \quad S'\subset S 
\end{equation}
 for any $S, S'\subset E$.
 
 We are  interested in two special sub-posets of $\cP(E)$, 
 written $\AGt$ and $\AGr$, 
 related to totally cyclic, respectively rooted,   orientations.
We saw that $\cO^0(G)\neq \emptyset$    (i.e. $G$  admits a totally cyclic  orientation)   only if $G$ is free from bridges.
We need to study  all totally cyclic orientations on all spanning subgraphs of $G$,
 so we    consider the following set
 $$
 \AGt:=\{S\subset E:  (G-S)_{br}=\emptyset\}.
 $$  
Next,   we know $\cO^1(G)\neq \emptyset$   (i.e. $G$  admits a rooted $1$-orientation) only  if $G$ is connected,
hence we set
 $$
\AGr:=\{S\subset E:   \  G-S  \text{ is connected}\}.
  $$  
  Of course, $\AGr$ is empty if $G$ is not connected.

  \begin{lemma}
 \label{rkBP} Let $b=0,1$ and assume $G$ connected if $b=1$.
 Then  $\AGb$ is a graded poset with respect to    \eqref{ri},  with rank  function mapping $S$ to 
 $ g(G-S).$
\end{lemma}
  In particular,
  $ \AGt$ has $E$  as unique minimal element   and  $G_{br}$ as unique maximal element, with 
 $ 
g(G-E) =\sum_{v\in V}w(v)$  and $g(G-G_{br})= g(G)$.
 If $G$ is connected, then 
$ \AGr$ has $\emptyset$  as  unique maximal  element, and  
its  minimal elements   are the    $S\subset E$ such that $G-S$ is a spanning tree. 

   \begin{remark}
\label{nobri} For any $S\subset E$ we have   $\At_{G-S}\ha  \AGt$.
If $S=G _{br}$ the injection  induces   an
 identification 
$ 
 \AGt=  \At_{G-G _{br}} . $ 
Indeed, for every $S\in  \AGt$ we have $G _{br}\subset S$,
 hence $S$ is also an element of $  \At_{G-G _{br}}$. \end{remark}

 \subsection{Posets   of  orientations.}
 
 We shall be considering  generalized orientations defined on various spanning subgraphs of  a fixed graph $G$. 
To   keep track of these subgraphs    we shall use subscripts, as follows.
 Given   $S\subset E$, we shall denote by
  $O_S$   a generalized orientation on  $G-S$. A generalized orientation with no subscript will be defined on the whole graph.
 
\begin{defi}
\label{podef}
Let $G$ be a graph and let $S,T\subset E(G)$. Given two   generalized orientations
 $O_S$ on $G-S$ and  $O_T$ on $G-T$ we set
 $$
 O_S\leq O_T \quad \text{  if } \quad  S\leq T \quad \text{and} \quad (O_{T})_{|G-S}= O _S.
$$
\end{defi}
It is easy to check that the above is a partial order.

We introduce, for a fixed graph $G$, the set of all totally cyclic orientations
on all spanning subgraphs of $G$. 
 \begin{equation}
\label{O1} 
 \OPt_G:=  \bigsqcup_{S\in  \AGt}{{\cO}^0}(G-S).
\end{equation}

Similarly, for rooted $1$-orientations
 \begin{equation}
\label{Ob} 
 \OPb_G:=  \bigsqcup_{S\in  \AGr}{{\cO}^1}(G-S).
\end{equation}
The notation ``$\OP$" indicates that  $  \OPt_G$ and $\OPb_G$ are endowed with the poset structure introduced 
in 
 Definition~\ref{podef}.

Finally, we   consider   orientations up to equivalence:

\begin{equation}
\label{O3}
\OGt  := \bigsqcup_{S\in  \AGt}{{\ocO}^0}(G-S) \quad  {\text{ and }}  \quad
\OGb  := \bigsqcup_{S\in  \AGr}{{\ocO}^1}(G-S).
\end{equation}
We will define a poset structure on $\OGt$ and $\OGb$.
We fix the following 
\begin{conv}
\label{convo}
Let $S\subset E(G)$ and consider the graph $G(S)$ defined in   \eqref{G(S)}.
Fix a  $b$-orientation, $\tilde{O}$, on $G(S)$. We identify $E(G(S))=E(\langle S \rangle)=S,
$ 
hence we can  define a  $b$-orientation, $\tilde{O}^*$, on $\langle S \rangle$ as follows.
Let $e\in S$. If $e$ is $\tilde{O}$-bioriented  then $e$ gets $\tilde{O}^*$-bioriented.
If   $e$ is not a loop  of $G(S)$ then   $e$ gets $\tilde{O}^*$-oriented according to $\tilde{O}$.
If $e$ is   a loop  of  $G(S)$   we choose an arbitrary orientation on $e$.
We refer to  $\tilde{O}^*$ as a  $b$-orientation {\it induced by $\tilde{O}$}.
\end{conv}

 \begin{lemma}
 \label{quoto}  Let $b=0,1$ and  
        $S,T\in  \AGb$   with $T\subset S$. Then for every
 $O_S\in \cO^b(G-S)$  there exists   ${O}_T\in \cO^b(G-T)$
  such that  ${O}_T\geq O_S$.

Moreover, if  $O _S\sim O'_{S}$ for some $O_{S}'\in \cO^b(G-S)$, there exists $ {O'}_T\in \cO^b(G-T)$  such that ${O'}_T\geq O_{S}'$
 and ${O'}_T\sim O_T$.
  \end{lemma}
\begin{proof} 
We first assume $b=0$.
Up to replacing $G$ with $G-T$, we can assume $T=\emptyset$
and $G$ bridgeless.
Hence   $G(S)$ is bridgless  and  we can fix  a   totally cyclic orientation,  $\tilde{O}$, on it.
Using   \ref{convo}, $\tilde{O}$ induces an orientation,  $\tilde{O}^*$, on $\langle S \rangle$.
Then $O_T:=O_S\cup \tilde{O}^*$  is an orientation on $G$. We claim $O_T$   is totally cyclic.
 By contradiction, let $F\subset E(G)$ be an $O_T$-directed cut of $G$. Then 
$F\cap E(G-S)=\emptyset$, as $G-S$ admits no $O_S$-directed cuts.
Therefore $F\subset S$, hence, using Lemma~\ref{ftriv} \eqref{ftrivcut}, $F$ is a directed cut of $G(S)$, which is not possible.
Finally, if $O_S\sim O_{S}'$, we construct  $O_{T}'$ using the same orientations $\tilde{O}$ and  $\tilde{O}^*$ used to construct  $O_T$.
Obviously,
$\md^{O_T}=\md^{O_{T}'}$,
hence   we are done.

The proof for $b=1$ follows the same steps. 
Up to replacing $G$ with $G-T$  we can assume $T=\emptyset$.
Now  $G(S)$ is bridgeless. Indeed, if $e\in S$ is a bridge of $G(S)$ it has to be a bridge of $G$, and hence $G-S$ is not connected, which is impossible by hypothesis.
We can thus fix  a   totally cyclic $0$-orientation,  $\tilde{O}$, on $G(S)$, and let
$\tilde{O}^*$ be a $0$-orientation on $\langle S \rangle$ induced by $\tilde{O}$.
Set  $O_T:=O_S\cup \tilde{O}^*$; arguing as for $b=0$ one checks that   $O_T$  is a rooted 1-orientation on $G$.
The rest of the proof  is the same as for $b=0$.
 \end{proof}

 \begin{prop}
 \label{poo}
Let $b=0,1$. Then
   $\OG$    is    partially ordered  as follows. 
For       $\oO_S $
 and  $\oO_T$ we
  set
$
\oO_S\leq \oO_T$   if $S\leq T$ and
if one of the two equivalent conditions below holds.
\begin{enumerate}[(i)]
\item
\label{exist}
There exist $O'_{S}\in \oO_S$ and  $O'_{T}\in \oO_T$ such that
$
(O'_{T})_{|G-S}= O' _{S}.
$
\item
\label{forall}
For every $O'_{S}\in \oO_S$ there exists  $O'_{T}\in \oO_T$  
such that
$
(O'_{T})_{|G-S}= O' _{S}.$
 \end{enumerate}
Moreover, the forgetful map, $
 \OG \to  \AGb$, sending $\oO_S$ to $S$,
 is a  quotient  of poset, and the map sending
  $ \oO_S$ to $g(G-S) $ 
is a rank   on    $\OG$.
 \end{prop}
\begin{proof} 
 Lemma~\ref{quoto} yields that \eqref{exist} implies  \eqref{forall}, and the converse is obvious.
Lemma~\ref{posettriv}  yields that we   have a partial order on $ \OG$.
The two forgetful maps are onto by Fact~\ref{F1} and Lemma~\ref{LmO}, and they are quotients by  Lemma~\ref{quoto}.
The rest of the statement is clear.  
 \end{proof}

\begin{remark}
 If $\oO_S\leq \oO_T$ then 
  $\md^{O_S} \leq   \md^{O _{T}} $, but the converse is not true. See Figure 1,
  where all vertices have weight $1$,   $T=\emptyset$ and $S$ consists of the bottom edge on the right of the first graph.
\begin{figure}[h]
\begin{equation*}
\xymatrix@=.5pc{
&&&&&&&&&&&&&&\\
&&&*{\bullet}\ar@{<-} @/^ 1.4pc/[rrrr]^(.02){2}\ar@{->}[rrrr]\ar@{<-} @/_1.4pc/[rrrr]_{ }&&&&*{\bullet} \ar@{<-} @/^ 1.2pc/[rrrr]^(1.1){3}\ar@{->} [rrrr]^(1){\;\; 1}\ar@{<-} @/_1.1pc/[rrrr]&&&&*{\bullet}   
&&&&  *{\bullet}\ar@{->} @/^ 1.4pc/[rrrr]^(.06){1}\ar@{<-}[rrrr]\ar@{->} @/_1.4pc/[rrrr]_{ }&&&&*{\bullet} \ar@{<-} @/^ 1.2pc/[rrrr]^(1.1){3}\ar@{->} [rrrr]^(1){\;\; 1} &&&&*{\bullet}    
  &&&  &&&&&& 
  \\
  \\
 &&&&&&&&\md^{O_T}& &&&&&&&&&&&\md^{O_S}&&&&&&&}
\end{equation*}
\caption{$\md^{O_S}\leq\md^{O_T}$ but $\oO_S\not\leq \oO_T$}
\end{figure}
 \end{remark}

Using Remark~\ref{nobri}  and similarly to it, we have
 \begin{remark}
\label{nobrio} 
For any $S\subset E$ we have  $\OPt_{G-S}\subset  \OPt_G$.
If $S=G _{br}$ we have two    
 identifications
$$
  \OPt_G=\OPt_{G-G _{br}}\quad \quad \quad \quad \OGt={\overline{\mathcal{OP}}^0}_{G-G _{br}}.
$$
\end{remark}
\begin{remark}
\label{noinjdeg} 
Consider the map
\begin{equation}
 \label{OGmap}
\OGt \la \Div(G);  \quad \quad O_S\mapsto \md^{O_S}. 
\end{equation} 
Its restriction to ${\overline{\mathcal{O}}^0}(G-S)$ is injective for every $S\in  \AGt$, 
yet,   the   map is not injective.
See   Figure 2, where $S$ and $T$ are    the dotted edges.
\begin{figure}[h]
\begin{equation*}
\xymatrix@=.5pc{
&&&&&&&&&&&&&&\\
&&&*{\bullet}\ar@{<-} @/^ 1.4pc/[rrrr]^(.02){1}\ar@{->}[rrrr]\ar@{.} @/_1.4pc/[rrrr]_{S}&&&&*{\bullet} \ar@{<-} @/^ 1.2pc/[rrrr]^(1.1){3}\ar@{->} [rrrr]^(1){\;\; 1}\ar@{<-} @/_1.1pc/[rrrr]&&&&*{\bullet}   
&&&&  *{\bullet}\ar@{->} @/^ 1.4pc/[rrrr]^(.06){1}\ar@{<-}[rrrr]\ar@{->} @/_1.4pc/[rrrr]_{ }&&&&*{\bullet} \ar@{<-} @/^ 1.2pc/[rrrr]^(1.1){3}\ar@{->} [rrrr]^(1){\;\; 1}\ar@{.} @/_1.1pc/[rrrr]_{T} &&&&*{\bullet} 
  &&&  &&&&&& 
 }
\end{equation*}
\caption{$\md^{O_S}=\md^{O_T}$ but $\oO_S\not\sim \oO_T$}
\end{figure}

\end{remark}

 \section{Functoriality under edge-contractions}
     \label{secfun}
In this section   we establish some combinatorial results, interesting on their own, needed in the algebro-geometric setting of Section~\ref{finsec}.
  As we shall see, there is a correspondence between edge-contractions and 
   degenerations of curves.
Therefore we here  study the  functorial behaviour    of generalized orientations
  with respect to edge-contractions.  
 \subsection{Contractions of stable graphs}
Recall that $\Sg$ denotes the set of stable graphs of genus $g$, and edge-contractions are defined  in Subsection~\ref{consec}. We begin with a simple   result, for which we use the notation  \eqref{G(S)}.

\begin{lemma}
\label{ftriv}
 Let   $S\subset E(G)$ and   $H:=G/S$.
 Let $T\subset E(H)$.
 Then
 
\begin{enumerate}[(a)]
  \item
  \label{ftriv2}
 $H-T =(G-T)/S$.
  \item
  \label{ftriv3}
  $H(T)=G(T)/S=G(T)$.
    \item
  \label{ftrivcut}
$T$ is a cut of $H$ if and only $T$  is a cut of $G$.
   \item
  \label{ftriv4}
 $H_{br}=\emptyset$ if and only if $G_{br}\subset S$.
 
\end{enumerate}
\end{lemma}
 
\begin{proof}
It suffices to assume $S=\{e\}$;   let $x,y\in V$ be the ends of $e$.
Denote by $v_e\in H$ the vertex to which $e$ is contracted; we have natural identifications
 $$
 E(H)=E(G)\ssm\{e\}\quad \text { and } \quad  V(H)=V(G)\cup \{v_e\} \ssm\{x,y\}.
 $$
Let us prove \eqref{ftriv2}. Using the above identities and the fact that $e\not\in T$, we   have   natural identifications
(viewed as equalities): 
 $$
 E(H-T)= E(H)\ssm T= E(G)\ssm (T\cup \{e\} )=E(G-T)\ssm  \{e\}=E\Bigr(\frac{G-T}{e}\Bigl)
 $$
 and, since   $V(H-T)=V(H)$
 $$
V(H-T) =V(G)\cup \{v_e\} \ssm\{x,y\}=V(G-T)\cup \{v_e\} \ssm\{x,y\}=V\Bigr(\frac{G-T}{e}\Bigl).
 $$
It is clear that the above identifications 
  induce a natural isomorphism
 between  $H-T$ and $(G-T)/e$.    \eqref{ftriv2} is proved. 
 
\eqref{ftriv3}.
We have
 $$
 H(T)=\frac{H}{ E(H)\ssm T }=\frac{G/e} {E(G)\ssm (e \cup T)}=\frac{G }{ (E(G)\ssm T)\cup  e}=\frac{G(T) }{e}=G(T).
 $$

\eqref{ftrivcut}.  By  \eqref{ftriv2} we have   $H-T=(G-T)/S$,  which  is connected if and only if $G-T$ is connected.   
  
\eqref{ftriv4}. Follows trivially from the preceeding parts.
\end{proof}

 For   two graphs, $G$ and $G'$,   we define the    {\it edge-contraction} relation:  
\begin{equation}
 \label{pograph}
 G'\geq G  \quad  \text {if } \quad  G'=G/S    \  \text{ for some } S\subset E(G).
\end{equation}
 Edge-contraction   is easily seen to be a partial order  on  the set of all graphs.

\begin{prop}
\label{rkSg}
 The set  $\Sg$, endowed with the edge-contraction relation defined in \eqref{pograph}, 
 is a graded poset with respect to the following rank 
 $$
 \Sg\la \N:\quad \quad G\mapsto  3g-3-|E(G)|.
$$
\end{prop}
\begin{proof}
It is well known that for every $G\in \Sg$
we have $|E(G)|\leq 3g-3$.

Let us prove that $\Sg$ is graded.
Let $G,H\in \Sg$ such that $H$ covers $G$.
Hence $H=G/S$ for some non empty $S\subset E(G)$.
We claim $|S|=1$. Indeed, if $|S|\geq 2$  there exists a non empty $S'\subsetneq S$.
But then by Remark~\ref{ftrivrk}    $G/S'\in \Sg$ and
$ 
H>G/S'>G,
$ 
a contradiction. Therefore $|S|=1$ and $|E(H)|=|E(G)|-1$ as wanted.
\end{proof}

\subsection{Bridgeless and connected subgraphs.}
We now study the behaviour of   $ \AGt$ and $ \AGr$, introduced in Subsection~\ref{brisec},  under edge-contractions.
Let {\sc graphs} be the category whose objects are graphs and whose morphisms are  contractions.
Let {\sc posets}  be the category whose objects are posets and whose morphisms are morphisms of posets.
For $b=0,1$ we have a map between  the objects of these categories,
\begin{equation}
\label{APP}
 \Ab: \{\text{{\sc graphs}}\} \la\{\text{{\sc posets}}\}; \quad G \mapsto \AGb.
\end{equation}
Using this map, we shall define  two functors from  {\sc graphs} to {\sc posets},
a covariant functor, written $(\Ab, \Ab_*)$, and a contravariant functor, written $(\Ab, \cA^{b*})$, so that $\Ab_*$ and $\cA^{b*}$ are
 the functor maps defined on morphisms.
 
\begin{lemma}
\label{forward}
Let $b=0,1$.   
For any    $\gamma:G\to H=G/S_0$ and   any $S\in    \AGb$ set
 $$
 \gamma_* S :=S\ssm  S_0.
 $$
 Then     the following hold:

\begin{enumerate}[(a)]
\item
 \label{flm0}
$ \gamma_* S \in   \Ab_H$.
 \item
 \label{flm1}
 If $T\in    \AGb $ is such that $S\leq T$, then $\gamma_*S\leq \gamma_*T$.
 \item
 \label{forward2}
Let ${ \delta}:H \to  J$ be a contraction of $H$. Then $({ \delta}\circ \gamma)_*={ \delta}_*\circ \gamma_*$.
\end{enumerate}
 In other words, the following is a covariant functor
$$
(\Ab, \Ab_*):\text{{\sc graphs}}\la  \text{{\sc posets}}
$$
 where 
  $\Ab_*(\gamma)(S)=  \gamma_*S$ for every  $\gamma:G\to H$ and $S\in \AGb$.
 \end{lemma}
\begin{proof}
We have, by Lemma~\ref{ftriv}\eqref{ftriv2}
 $$
 H- \gamma_*S= H-(S\ssm  S_0) =\frac{G-  (S\ssm  S_0)}{S_0}.
 $$
If $b=0$  we
  must check   $H- \gamma_*S$ has no bridges. As 
  $G-S$ has no bridges any bridge of $G-  (S\ssm  S_0)$ must lie in $S_0$, hence 
  its quotient by $S_0$ is bridgeless, and we are done.
If  $b=1$ we
  must prove   $H- \gamma_*S$ is connected.
As $G-S$ is connected so is $G-  (S\ssm  S_0)$, hence so is its quotient. \eqref{flm0} is proved.
 
 \eqref{flm1} and \eqref{forward2} are obvious.
\end{proof}
Recall that $ \AGt$ and $ \AGr$ are   graded posets. 
Now, the map  $\gamma_*$ does not preserve the gradings. Indeed, let  $e\in E(G)\ssm G_{br}$.
  Set $S=S_0=\{e\}$ so that $\gamma_*S=\emptyset$.
  We have $g(G-S)=g(G)-1$  and $g(H-\gamma_*S)=g(H)=g(G)$.
 By contrast,  the ``pull-back" map, with the associated contravariant functor, defined below,  does preserve the grading.
 
\begin{lemma}
\label{flm}
Let $b=0,1$.   
For any   $\gamma:G\to H=G/S_0$ and   $T\in  \Ab_H$ 
define   $ \gamma^* T \subset E(G)$   as follows
\begin{equation}
 \label{}
 \gamma^* T: =\begin{cases}
T\cup(G-T)_{br} & \text{ if } b=0\\
T &  \text{ if }b=1. 
\end{cases}
\end{equation}
  
 Then     the following hold:
\begin{enumerate}[(a)]
 \item
  \label{flm0}
 $ \gamma^* T \in  \AGb$ and   $g(H-T)=g(G-\gamma^*T)$.
 \item
 \label{flm1}
 If $R\in   \Ab_H$ is such that $R\leq T$, then $\gamma^*R\leq \gamma^*T$.
 \item
 \label{flm2}
Let $\delta:H \to  J$ be a contraction of $H$. Then $(\delta\circ \gamma)^*=\gamma^*\circ\delta^*$.

   \end{enumerate}
    In short, the following is a grading-preserving,   contravariant functor
$$
(\Ab, \cA^{b*}):\text{{\sc graphs}}\la  \text{{\sc posets}}
$$
 where 
  $\cA^{b*}(\gamma)(T)=  \gamma^*T$ for every   $\gamma:G\to H$ and $T\in \Ab_H$.
   \end{lemma}
 
\begin{proof}
The only nontrivial claim of \eqref{flm0} is the last, i.e.    that $\gamma^*$ preserves the rank.
We provide the proof in case $b=0$, which  trivially gives also  the proof for   $b=1$.
$$
g(H-T)=g\Bigr(\frac{G-T}{S_0}\Bigl)=g(G-T)=g\Bigr((G-T)-(G-T)_{br}\Bigl)=g (G-\gamma^*T), 
$$
 where we used Lemma~\ref{ftriv}\eqref{ftriv2} in  the first equality,
and  that  contractions and bridge-removals   preserve the genus in the second and third equality.
 
\eqref{flm1} is obvious if $b=1$. Let   $R\in   \At_H$   such that  $T\subset R$. We must prove $\gamma^*T\subset \gamma^*R$.
It is clearly enough  to prove 
 $ 
 (G-T)_{br}\subset (G-R)_{br}.
 $ 
 
Since $(H-T)_{br}=\emptyset$  and, by  Lemma~\ref{ftriv}\eqref{ftriv2},  $H-T=(G-T)/S_0$,  we have $(G-T)_{br}\subset S_0$.
 Hence $(G-T)_{br}\cap R=\emptyset$. Therefore, as $G-R\subset G-T$, we have $ (G-T)_{br}\subset (G-R)_{br}$ as wanted.
 
 We omit the direct proof of \eqref{flm2},  which   follows   easily from \ref{fupr}\eqref{fupr2}.
 \end{proof}
 
\begin{prop}
 \label{fupr} Let $b=0,1$.  
 Fix   a contraction $\gamma:G\to H=G/S_0$. Let  $S\in  \AGb$ and $T\in    \Ab_H$.  
Then
\begin{enumerate}[(a)]
 \item
  \label{fupr1}
 $
 \gamma_*  \gamma^* T=T$ (equivalently, $\Ab_*(\gamma)\cA^{b*}(\gamma)=\operatorname{id}_{\Ab_H}$).
  
  \item
  \label{fuprb}
$T\subset \gamma_*S  \Leftrightarrow \gamma^*T\subset S$.
\item
  \label{fupr2}
$\gamma^*T$ is the smallest (by inclusion) element of $ \AGb$ whose image under $\gamma_*$  equals $T$.
   \item
 \label{fupr3}
$\Ab_*( \gamma): \AGb \to \Ab_H  $ is   a quotient of  posets.
  \item
 \label{fupr4}
 If $S_0\subset G_{br}$ then $\Ab_*( \gamma): \AGb \to \Ab_H  $is an isomorphism.
\end{enumerate}
 \end{prop}
\begin{proof}
\eqref{fupr1}, \eqref{fuprb} and \eqref{fupr2} are  obvious if $b=1$, so
assume $b=0$.  We have 
 $ 
 \gamma_*  \gamma^* T=\gamma_*(T\cup (G-T)_{br})=\bigr(T\cup (G-T)_{br}\bigl) \ssm S_0.
 $ 
 By hypothesis $(H-T)_{br}$ is empty, hence, by Lemma~\ref{ftriv}, $(G-T)_{br}\subset S_0$.
 Therefore 
 $$
  \gamma_*  \gamma^* T=\bigr(T\cup (G-T)_{br}\bigl) \ssm S_0= T \ssm S_0=T.
 $$
 \eqref{fupr1} is proved. 
The implication $\Leftarrow$ in \eqref{fuprb} follows trivially from  \eqref{fupr1}.
For the other implication, the hypothesis is
  $T \subset S\ssm S_0$, hence $T\subset S$. 
  Since  $\gamma^*T= T\cup(G-T)_{br}$ it is enough to prove $(G-T)_{br}\subset S$.
We have $G-S\subset G-T$, hence   every bridge of $G-T$ is either contained in $S$, or a bridge of $G-S$. As $G-S$ is bridgeless,
we   conclude $(G-T)_{br}\subset S$.

  \eqref{fupr2} follows immediately from \eqref{fuprb}.
  
\eqref{fupr3}.    Part \eqref{fupr1} implies $\Ab_*(\gamma)$ is surjective 
   and, for any $T,T'\in   \At_H$, we have 
 $T=\gamma_*  \gamma^* T$ and $T'=\gamma_*  \gamma^* T'$. By Lemma~\ref{flm}, if $T\leq T'$ then 
 $ \gamma^* T\leq \gamma^* T'$. Hence we are done.
 
\eqref{fupr4}. 
Notice that 
$\cA^{b*}(\gamma)$ is obviously injective. If $S_0$ is made of bridges of $G$ then $S_0\subset S$ for any $S\in  \AGt$,
 and $S\cap S_0=\emptyset$ for any $S\in  \AGr$. Hence $\Ab_*(\gamma)$ is injective,
 and we are done.
\end{proof}

 \subsection{Direct image of divisors and orientations}
 In this subsection we will denote by $\gamma:G\to G/S_0=H$ a contraction, with $S_0\subset E(G)$.
 To any  contraction $\gamma$  we associate    a map, easily checked to be a surjective group homomorphism,
  from $  \Div(G)$ to $\Div (H)$ mapping $\md$ to $\gamma_*\md$
defined as follows
$$
 (\gamma_* \md)_v :=\sum _{z\in \gamma_V^{-1}(v)}\md_z 
$$
for any  $v\in V(H)$. 
Let ${ \delta}:H \to  J$ be a contraction. Then
\begin{equation}\label{homo}
 ({ \delta}\circ \gamma)_*(\md)={ \delta}_* (\gamma_* ( \md)).
\end{equation}

In the sequel we shall employ the following   notation.
Let $O$ be a generalized orientation on $G$ and let $\gamma:G\to H$ be a contraction.
As $E(H)$ is identified with a subset of $E(G)$ we can restrict $O$ to $E(H)$, thus defining a generalized orientation on $H$, 
 denoted by $O_{|H}$.   

Let  $S\subset E$  and  let   $O_S$  be a    generalized orientation 
 on  $G-S$.  We have  $ 
E(H-\gamma_*S)=E(G -S\cup S_0)\subset E(G-S),
$ so we can define (abusing notation again) the following  generalized  orientation on
 $H-\gamma_*S$
\begin{equation}
  \label{defio}
  \gamma_* O_S  :=(O_S)_{|H-\gamma_*S}. 
\end{equation}

As a final piece of notation, to  $\gamma$ and $S\subset E$
we associate the divisor $\mc^{\gamma,S}$ on $H$ 
such that  for any $v\in V(H)$
\begin{equation}
 \label{mcv}
 \mc_v^{\gamma,S}:=|\{e\in S_0\cap S: \  \gamma(e)=v\}|.
\end{equation}
If $S=E(G)$ we write   $\mc^{\gamma} =\mc^{\gamma,E(G)}$. Of course, $ \mc^{\gamma,S}\geq 0$ and equality
holds if and only if $S\cap S_0=\emptyset$.

\begin{prop}
\label{fprop}
Let $G$ be a graph,  $S\subset E$, and 
  $O_S$   a $b$-orientation on  $G-S$, with $b=0,1$.
 Let   $\gamma:G\to H=G/S_0$ be a contraction 
 such that no edge of $S_0$ is bioriented. 
Then  $\gamma_*O_S$ is a $b$-orientation on $H-\gamma_*S$ and the following hold.
\begin{enumerate}[(a)]
 \item
  \label{fprop1}
If $O_S\in \cO^b(G-S)$ then
 $\gamma_*O_S \in \cO^b(H-\gamma_*S)$.
  \item
 \label{ffpropfunc}
Let ${ \delta}:H \to  J$ be a contraction of $H$. Then $({ \delta}\circ \gamma)_*O_S={ \delta}_* \gamma_* O_S $.
  \item
 \label{dlm}
 $\gamma_*\md^{O_S}= \md^{\gamma_*O_{S}}- \mc^{\gamma,S}$.
    \item
   \label{fprop3}
Let 
  $O'_{S}$ be a $b$-orientation on $G-S$.  If $O'_S\sim O_S$ then  $\gamma_*O'_S\sim   \gamma_*O_S$.
 \item
  \label{fprop4} Let 
  $O_{T}$ be a $b$-orientation on $G-T$.
If $ O_S  \leq  O_{T} $ then $ \gamma_*O_S \leq  \gamma_*O_{T} $.

       \end{enumerate}
\end{prop}

\begin{proof}
It is clear that $ \gamma_* O_S $ is   a  $b$-orientation on
 $H-\gamma_*S$ whose bioriented edge, in case $b=1$, is the same as that of $O_S$. 

\eqref{fprop1}. We need to show  $\gamma_* O_S $  is totally cyclic if $b=0$, and rooted if $b=1$.  
  It suffices to prove that if $F$ is a directed cut of $H -\gamma_*S$
then $F$ is a directed cut of $G-S$. We can assume $S_0=\{e_0\}$. 
If $e_0\not\in S$  then $\gamma_*S=S$. By Lemma~\ref{ftriv} \eqref{ftrivcut}, every directed cut   of $H-S$ is also a directed cut of $G-S$ and we are done. 
If $e_0 \in S$  set $T=S\ssm \{e_0\}$. We have
$$
H -\gamma_*S=H -T= (G-T)/e_0.
$$
 A directed cut, $F$,    of $H -\gamma_*S$ is thus a directed cut
of $G-T$.
Now, $G-S\subset G-T$, hence $F$ is a directed cut in $G-S$.
 \eqref{fprop1} is proved.

 \eqref{ffpropfunc} is trivial.
 
  \eqref{dlm}.  For any $v\in V(H)$ 
set $Z_v=\gamma^{-1}(v)$, which is   a connected subgraph of $G$.  
We have
 $ g(Z_v)=\sum_{z\in V(Z_v)} \bigr(w(z)-1 \bigl)+|E(Z_v)| +1$,   
hence
  $$
  (\gamma_*\md^{O_S})_v= \sum_{z\in V(Z_v)}\bigr(w(z)-1+\mt_z^{O_S}\bigl)=g(Z_v)-1-|E(Z_v)|+\sum_{z\in V(Z_v)}  \mt_z^{O_S}.
  $$
Let  $t^{O_S}(Z_v)$ be the number of edges  
with  target in $Z_v$ and not  contained in it. As every edge of $Z_v$ lies in $S_0$,
$$
|E(Z_v)|  = \sum_{z\in V(Z_v)}   \mt_z^{O_S} - t^{O_S}(Z_v)  +   \mc_v^{\gamma,S}.
 $$
Therefore
\begin{equation}
 \label{eq*}
  (\gamma_*\md^{O_S})_v =  g(Z_v) -1 + t^{O_S}(Z_v) -    \mc_v^{\gamma,S}.
\end{equation}
   On the other hand we have 
\begin{equation}
 \label{eqzz}
  (\md^{\gamma_*O_S})_v= w_{/S_0}(v)-1+\mt_v^{\gamma_*O_S} = g(Z_v)-1+  t^{O_S}(Z_v).
\end{equation}
Indeed, by definition of contraction, $w_{/S_0}(v)= g(Z_v)$ and, clearly,
 the number of $O_S$-incoming edges at $Z_v$ equals the number of   $\gamma_*O_S$-incoming edges at $v$.
Comparing \eqref{eq*} and \eqref{eqzz} yields \eqref{dlm}.

  \eqref{fprop3}.  By hypothesis,
 $\md^{O_S}=\md^{O'_{S}}$, hence  $\mt^{O_S}=\mt^{O'_{S}}$.
Hence, by \eqref{tOZ}, for any $v\in V(H)$ we have $t^{O_S}(Z_v)= t^{O'_{S}}(Z_v)$ as $Z_v$ does not contain bioriented edges.
Combining with 
   \eqref{eqzz} we get  $\md^{\gamma_*O_S}=\md^{\gamma_*O'_{S}}$, and we are done.

 \eqref{fprop4}. By   assumption we have $S\leq T$  and 
   $ (O_{T})_{|G-S}= O _{S}$.
We obviously have
$ \gamma_*S\leq \gamma_*T 
$. Next, as  $H- \gamma_*S\subset H-\gamma_*T 
$
 $$
(\gamma_*O_{T})_{|H-\gamma_*S} = (O_T)_{|H-\gamma_*S}=
 (O_T\:_{|G-S})_{|H-\gamma_*S}=O _{S}\;_{|H-\gamma_*S}= \gamma_*O _{S}. 
$$
The proof is complete
\end{proof}

\begin{example}
In Figure 3 we have  $S=S_0=\{e\}$.

 \begin{figure}[h]
\begin{equation*}
\xymatrix@=.5pc{
&&&&&&&&&&&&&&\\
&G=&*{\bullet}\ar@{<-} @/^ 1.4pc/[rrrr]^(.02){}\ar@{->}[rrrr]\ar@{.} @/_1.4pc/[rrrr]^{e}&&&&
*{\bullet} \ar@{->} @/^ 1.6pc/[rrrr]^(1.1){}\ar@{<-} [rrrr]^(1){}\ar@{->} @/_1.5pc/[rrrr]&&&&*{\bullet}   
 &\ar@{->}[rrr]
&&  &&H=&&*{\bullet}\ar@{->}@(ul,ur)\ar@{->}@(ul,dl)
 \ar@{->} @/^ 1.2pc/[rrrr]\ar@{<-} @/_0.1pc/[rrrr]\ar@{->} @/_1.5pc/[rrrr]_(0,1){v_e}&&&&*{\bullet}    
  &&& &&&&&&& 
  \\
  \\
&&&&& O_S &&&&&&&&&&&&&&\gamma_*O_S&&&&&&}
\end{equation*}
\caption{Case $S=S_0$}
\end{figure}
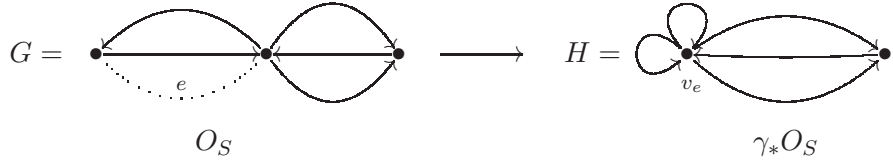
Assume all vertices of $G$ have weight $1$, so that $v_e$ has weight $2$ in $H$.
We have, ordering the vertices from left to  right,
$ 
\mt^{O_S}=\md^{O_S} = (1,2,2)$, \    $\mt^{\gamma_*O_S}=(3, 2),$ \  $   \md^{\gamma_*O_S}=(4, 2)$, and 
$ 
 \gamma_*\md^{O_S} = (3,2).  
$ 
Hence $\md^{\gamma_*O_S}>\gamma_*\md^{O_S}$. 
 \end{example}

From the previous result  we derive a few   facts.
\begin{prop}
\label{fbar} 
Fix   $\gamma:G\to H=G/S_0$   and let $b=0,1$.
\begin{enumerate}[(a)]
 \item
Let $b=0$. Then 
we have   a   morphism  of posets
$$ 
 {\overline {\gamma}}_*:\OGt\la  \OHt; \quad \oO_S\mapsto \overline {\gamma_*O_S}.
$$
 \item
Let $b=1$  and $S_0\neq E(G)$.
Then we have   a   morphism  of posets
$$ 
 {\overline {\gamma}}_*:\OGb\la  \OHb; \quad \oO_S\mapsto \overline {\gamma_*O'_{S}}
 $$ 
 for any $O'_{S}\sim O_S$ whose bioriented edge is not in $S_0$.
 \item
Let $b=0,1$ and let
  ${ \delta}:H \to  H/T_0$ be  a contraction;    if $b=1$ assume $T_0\neq E(H)$.
Then $\overline{({\delta}\circ \gamma)}_*={\overline {\delta}}_*\circ {\overline {\gamma}}_* $.

\end{enumerate}
 \end{prop}
 
 \begin{proof}
 If $b=0$ the statement is a trivial consequence of \ref{fprop}.
 
For $b=1$, pick any   $\oO_S\in \OGt$.  By Lemma~\ref{lmfree}, there exists $O'_{S}\sim O_S$ whose bioriented edge
does not lie in $S_0$. Then  \ref{fprop} yields that $\gamma_*O'_{S}$ is a well-defined element in
$\OPt_H$,   and different choices of $O'_{S}$ yield equivalent elements in $\OPt_H$. Hence $ {\overline {\gamma}}_*\oO_S$ is a well defined
element of $\OHt$.
The rest of the proof follows   from \ref{fprop}.
  \end{proof}

 \begin{cor}
\label{fdiag}
 Let $\gamma:G\to H=G/S_0$ be a contraction. Then  we have   a commutative diagram of posets
$$\xymatrix{
\OPt_G\ar[rr]^{\gamma_*}\ar[d]_{} &&
\OPt_H\ar[d] \\
\OGt\ar[rr]^{\ov{\gamma}_*} && \OHt,} $$ 
where  
the vertical arrows are the quotient maps.
If $S_0\subset G_{br}$ then the horizontal arrows  are bijections.\end{cor}

\begin{remark}
If $S_0\subset G_{br}$ the lower arrow,  $\ov{\gamma}_*$, is a  bijection also for $b=1$. The proof uses a different language so we omit it as we will not need it.
\end{remark}
  \begin{proof}
 The commutativity of the diagram follows from Propositions~\ref{fprop} and \ref{fbar}.
 For the remaining  part  it is enough to prove that  $\gamma_*$   is a bijection.
 
We have $S_0\subset S$ for all $S\in  \AGt$ and we already know we have  a bijection $ \AGt\to   \At_H$
 mapping $S$ to $\gamma_*S$.
Now $G-S$ and $H-\gamma_*S$ have exactly the same edges, hence we have an injection
 $\gamma_*:\cO^0(G-S)\ha \cO^0(H-\gamma_*S)$. We proved that $\gamma_*$ is injective.
Now pick $O_T\in \cO(H-T)$. Let $S=\gamma^*T=T\cup(G-T)_{br}$ so that $\gamma_*S=T$. We have 
$(G-T)_{br}\subset S_0$ hence
$$
E(G-\gamma^*T)=E(G)\ssm \bigr(T\cup(G-T)_{br}\bigl)\subset E(G)\ssm (T\cup S_0)=E(H-T).
$$
Therefore we can restrict $O_T$ to  $G-\gamma^*T$,  obtaining an orientation    easily seen to be totally cyclic and to map to $O_T$ via $\gamma_*$.
Hence $\gamma_*$ is surjective.
 \end{proof}

\begin{cor}
\label{bricor}
The  inclusion $\iota:G-G _{br}\ha G$ and   the contraction  $\gamma: G\to G/G _{br}$
induce natural  isomorphisms (viewed as identifications) 
 $$
\OPt_{G-G _{br}} \stackrel{\iota_*}{=}\OPt_G\stackrel{\gamma_*}{=}{\OPt_{G/G _{br}}  }  
$$
and
$${\overline{\mathcal{OP}}^0}_{G-G _{br}} \stackrel{\ov{\iota}_*}{=}{\overline{\mathcal{OP}}^0_G}\stackrel{\ov{\gamma}_*}{=}{\overline{\mathcal{OP}}}^0_{G/G _{br}}. 
$$

\end{cor}
 
\begin{proof}
Combine Remark~\ref{nobrio} with Corollary~\ref{fdiag}.
\end{proof}
 \subsection{Quotients of orientation spaces}
We shall now  give a more precise description of     the map  $  
{\overline {\gamma}}_*:\oOPGb\to  \oOPHb 
$   introduced in Proposition~\ref{fbar}.    
\begin{thm}
\label{fthm}
 Let $\gamma:G\to H=G/S_0$ be a contraction with $S_0\subsetneq E(G)$; let $b=0,1$.
Then   $  
{\overline {\gamma}}_*:\oOPGb\to \oOPHb 
$ 
  is a quotient of posets mapping
  $\ocO^b(G-\gamma^*T)$
onto
 $ 
  \ocO^b(H-T) 
 $ for every $T\subset E(H)$.
\end{thm}

We begin with the case $b=0$, for which we have the following.

\begin{prop}
\label{fsu}
 A contraction $\gamma:G\to H$ induces the  quotient of posets  
  $\gamma_*: \OPt_G\to \OPt_H$  mapping
  $\cO^0(G-\gamma^*T)$
onto
 $ 
  \cO^0(H-T) 
 $ for every $T\in    \At_H$.
\end{prop}
\begin{proof}
We proceed in  three steps. Steps 1 and 2 prove that $\gamma_*$ is a quotient, Steps 1 and 3 prove that it is onto as stated.

 \
 
\noindent {\bf Step 1}. {\it Suppose  $G_{br}=\emptyset$,
 then the restriction of  $\gamma_*$ to $ \cO^0(G)$
 gives a surjection
 $ 
 \cO^0(G)\twoheadrightarrow \cO^0(H).
 $ }

 We can assume $S_0=\{e\}$.
 As $G_{br}=\emptyset$ we have $H_{br}=\emptyset$. Fix $\tilde{O}\in \cO^0(H)$.
 If $e$ is a loop or if $H$ has only one vertex the statement is trivial, so we exclude this
 and let   $x,y\in V(G)$ be the ends of $e$.  
Now, using convention~\ref{convo},
we have an orientation $\tilde{O}^*$  on  $G-e$ induced by $\tilde{O}$.
We shall denote $O_{e}=\tilde{O}^*$ 
and  prove that we can extend $O_{e}$ to $e$ by a totally cyclic orientation on $G$,
written $O$. Obviously,  we will have $\gamma_*O=\tilde{O}$.
  
We denote $v_e=\gamma(e)$.
Since $\tilde{O}$ is totally cyclic  
we can fix a cyclically oriented cycle $ {C}\subset H$ containing $v_e$.
Then it is easy to check that   the edges of $ {C}$ generate in $G$
a subgraph, $P:=\langle  E( {C}) \rangle$, which is an $O_e$-directed  path having $x$ and  $y$
as ends.
Of course, $P$ does not contain $e$, hence $C_e:=P+e$ is a cycle in $G$.
We now orient $e$ in such a way that $C_e$ becomes  a  cyclically oriented cycle.
This gives an orientation, $O\geq O_e$, on  $G$, which we claim is totally cyclic.
Indeed, let $F\subset E(G)$ be an $O$-directed cut.
Then $e\in F$ (for otherwise $F$ would be a $\tilde{O}$-directed cut of $H$). Hence $F\cap E(C_e)\neq \emptyset$, and hence
  $F\cap E(C_e)$ is a directed cut of the cyclically oriented cycle $C_e$. This is not possible.
  Step 1 is proved.
  
 \
  
   \noindent {\bf Step 2}. {\it  Let $O_T, O_R\in  \OPt_H$ with $O_T\geq O_R$.
   Then there exist  $O_{\gamma^*T},O_{\gamma^*R}\in   \OPt_G$  such that $\gamma_*O_{\gamma^*T}=O_T$,
    $\gamma_*O_{\gamma^*R}=O_R$ and  $O_{\gamma^*T}\geq O_{\gamma^*R}$.}

By hypothesis $T\geq R$, hence    $G-\gamma^*T\supset G-\gamma^*R$.
We   assume $S_0=\{e\}$ and we use the same set-up of Step 1.

We begin by fixing  a totally cyclic    orientation $O_{\gamma^*R}$  induced by $O_{R}$ as described in Step 1.
To define $O_{\gamma^*R}$   on $G-\gamma^*R$ 
 the only choices  we  make
are for    non-loop edges   corresponding to   loops  of $H-R$
(the orientation is chosen arbitrarily, see \ref{convo}), and for the contracted edge $e$, if  $e\in G-\gamma^*R$
(the orientation is chosen to ensure total  cyclicity).

Now,  among all orientations induced by $O_T$ on $G-\gamma^*T$ according to \ref{convo}, we choose one,
written $O_{\gamma^*T}$,
 with the   requirement that it   agrees with $O_{\gamma^*R}$ on $G-\gamma^*R$.
 Hence every non loop-edge    corresponding to a loop  of $H-R$,
 is oriented in the same way as   in  $O_{\gamma^*R}$ and,  more importantly,
if the contracted edge $e$ is contained in  $G-\gamma^*R$  then it has to be  $O_{\gamma^*T}$-oriented
as in $O_{\gamma^*R}$.

 Obviously, $O_{\gamma^*T}\geq O_{\gamma^*R}$. We need to check $O_{\gamma^*T}$ is totally cyclic.
 By construction,  we  need to prove it only  in case $e\in G-\gamma^*R$
 (in the other case the $O_{\gamma^*T}$-orientation on $e$ is given as in Step 1, to ensure $O_{\gamma^*T}$ is totally cyclic).
  By contradiction, let $F$ be a directed cut of $G-\gamma^*T$. Then $e\in F$, for otherwise $F$ would be a cut
  of $H-T$. Hence  $F\cap E( G-\gamma^*R)$  is not empty. Hence $F$ induces a directed cut of $G-\gamma^*R$,
  which is not possible.

 \
  
     \noindent {\bf Step 3}. {\it The restriction of  $\gamma_*$ to $ \cO^0(G-\gamma^*T)$
is a surjection  onto $ \cO^0(H-T). $}
 
 We shall reduce this   to Step 1, to do which  we need to handle the problem that
 $(G-\gamma^*T)/S_0$ may fail to be equal to $H-T$.
 
Consider the   contraction  induced by restricting $\gamma$ to $G-T$
 $$ \gamma_{|G-T}:G-T \la (G-T)/S_0=H-T $$ 
   (using  \ref{ftriv} \eqref{ftriv2}).
We have  $(G-T)_{br}\subset S_0$, hence we can 
   factor $\gamma_{|G-T}$ 
  $$
  \gamma_{|G-T}:G-T
   \la (G-T)/(G-T)_{br} \stackrel{\gamma'}{\la} (G-T)/S_0=H-T.
  $$
  Set $$J:=\frac{G-T}{(G-T)_{br}}, \quad \quad \tilde{J}:=\frac{J}{S_0- (G-T)_{br}}=H-T.
  $$
As $J$ is bridgeless  we can apply
the conclusion of Step 1 to the contraction $\gamma':J\to \tilde{J}$. Hence   $\gamma'_*$
 yields a surjection 
\begin{equation}
 \label{OOHH}
\cO^0(J) \la \cO^0(\tilde{J})= \cO^0(H-T).
\end{equation}
On the other hand we have natural identifications   
$$
\OPt_J
=\OPt_{G-T}= 
\OPt_{(G-T)-(G-T)_{br}} 
=  \OPt_{G-\gamma^*T}
$$
   using \ref{bricor} for the first two  equalities.
Combining   with \eqref{OOHH} we obtain the surjection
 $ 
\OPt_{G-\gamma^*T} \to  \OPt_{H-T}.
$ 
Step 3 and the Proposition  are proved.
 \end{proof}
 
 \noindent
 {\it {Proof of Theorem~\ref{fthm}.}}

The case $b=0$ follows from Proposition~\ref{fsu}. Suppose $b=1$.
 We argue similarly  to the proof of Proposition~\ref{fsu}. We begin by proving that $\ov{\gamma}_*$ induces   a surjection
$  \ocO^1(G)\twoheadrightarrow\ocO^1(H). $
 We can assume $G$ and $H$ connected, and
   $S_0=\{e\}$; we write  $v_e=\gamma(e)$ and $x,y\in V(G)$ for the ends of $e$ ($x\neq y$ otherwise we are done).

 Fix $\tilde{O}\in \cO^1(H)$, then, by \ref{convo}, 
we  have a  $1$-orientation $O_{e}= \tilde{O}^*$  on  $G-e$ induced by $\tilde{O}$.
We shall   prove   we can extend $O_{e}$    by a rooted orientation, $O$, on $G$,
whose bioriented edge is the same as that of $\tilde{O}$,  denoted by $\tilde{e}$.

  As $\tilde{O}$ is rooted, there exists a directed path $\tilde{P}\subset H$ from 
$\tilde{e}$ to $v_e$. It is clear that the edges of $\tilde{P}$ span in $G$ a directed path, $P$, from $\tilde{e}$
to $x$ (say) and not containing $e$. We set $P_e=P+e$ and orient $e$ so that $P_e$ is a directed path from $\tilde{e}$
to $y$. Let $O$ be the so-obtained orientation on $G$; we shall prove it is rooted using  Lemma~\ref{lmfree} \eqref{lmfreeb}.

Let $w\in V(G)$, we must exhibit an $O$-directed path from $\tilde{e}$ to $w$. If $w=x,y$ it suffices to take $P$  or $P_e$.
So we can assume $w$ is also a vertex of $H$ different from $v_e$. Let $\tilde{P}_w\subset H$ be a directed path from  $\tilde{e}$
to $w$. If $\tilde{P}_w$ does not contain $v_e$ then $\tilde{P}_w$ is naturally identified with a directed path in $G$ from $\tilde{e}$
to $w$ and we  are done. If $v_e$ is in  $\tilde{P}_w$,   we can write
$ 
\tilde{P}_w=\tilde{Q_1}+\tilde{Q_2}
$ 
where $\tilde{Q_1}$ is a directed path from $\tilde{e}$ to $v_e$ and  $\tilde{Q_2}$ is a directed path from $v_e$ to $w$ not containing $v_e$. Hence $\tilde{Q_2}$ corresponds to a directed path, $Q_2$, from either $x$ or $y$ to $w$.
In $G$, we attach  $Q_2$ to either $P$ (if $Q_2$ starts at $x$) or $P_e$ (if $Q_2$ starts at $y$)
getting 
  a path in $G$ directed  from  $\tilde{e}$ to $w$. 

 We conclude that the restriction of  $\ov{\gamma}_*$ to 
 $  \ocO^1(G)$ surjects onto $\ocO^1(H). $
 
 The rest of the  proof   is the same as  for  Proposition~\ref{fsu}, Steps 2 and 3, mutatis mutandis.  
 Theorem~\ref{fthm} is proved.
 $\qed$
 
\subsection{Orientations in genus $g$}
 We use     notation  \eqref{APP}.
 
\begin{defi}
\label{Bgdef}
Let $g\geq 2$ and let $b=0,1$. Set
 $$
 \Ag:=\{(G, S):\  G\in \Sg, \  S\in  \AGb\}
 $$
and endow it  with the following partial  order relation:
 $$(G, S)\leq  (H, T)\quad  \quad\quad  \text{ if }   \quad \quad G\leq H  \quad \text{ and }   \quad S  \leq \gamma^*T $$
for some (possibly trivial) contraction 
 $\gamma:G\to H$.
 \end{defi}
 It is easy to check that $ \Ag$ is indeed a poset inducing, for every $G\in \Sg$,
 the poset structure on $ \AGb$ defined earlier.

\begin{prop}
 \label{Bgq}
Let $g\geq 2$ and let $b=0,1$.
\begin{enumerate}[(a)]
 \item
 \label{Bgqa}
The   map 
$ 
\Ag\la \Sg 
$ mapping $(G,S)$ to $G$
  is a quotient of posets.
\item
\label{Bgqb}
The following is a rank  on   $ \Ag$
 $$
 \rho_{\Ag}: \Ag \la \N;\quad \quad (G, S)\mapsto 3g-3-|E(G)|+g(G-S).
 $$
 \end{enumerate}
\end{prop}
  
\begin{proof}
The   map in \eqref{Bgqa}  is  clearly a surjective morphism of posets. To check that it is a quotient, pick $G,H\in \Sg$
with $G\leq H$. Fix $T\in   \Ab_H$, then $\gamma^*T\in  \AGb$ and, of course,
  $(G, \gamma^* T)\leq (H,T)$. \eqref{Bgqa} is proved.
  
Write $\rho=\rho_{\Ag}$. Fix $(G,S)$ and $(H,T)\in \Ag$   such that $(H,T)$ covers  $(G,S)$.

First, suppose $G\neq H$.  
We claim      $S=\gamma^*T$. By contradiction, suppose   $S< \gamma^*T$.
Then  
$ 
(G,S) <  (G,\gamma^*T) < (H,T),
$ 
a contradiction.  
Hence $S=\gamma^*T$. But then 
$G$ covers $H$ in $\Sg$, indeed if $G<G'<H$ for some $G'\in \Sg$
then 
$ 
(G,S) <  (G', T') < (H,T),
$ 
where $T'$ is the pull-back of $T$ to $G'$ under the contraction $G'\to H$; this is impossible.
As $H$ covers $G$,    Proposition~\ref{rkSg}  gives $|E(G)|=|E(H)|+1$,  hence
$$
\rho (G,S)-(3g-3)= g(G-S)-|E(G)|=g(G-\gamma^*T)-|E(H)|-1.
$$
Now,   Lemma~\ref{flm} \eqref{flm0} yields $g(G-\gamma^*T)=g(H-T)$, hence
$$
\rho (H,T)-\rho(G,S)=g(H-T)-|E(H)| - (g(H-T)-|E(H)|-1)= 1.
$$
As wanted.
Now, suppose $G=H$. Then  $\gamma^*T=T$ and   $T$ covers $S$  
(for otherwise we would have $(G,S)<  (G, S') < (G,T)$ for   $S'$ between $S$ and $T$).
By Lemma~\ref{rkBP} we have 
$ 
g(G-S)= g(G-T)-1=g(H-T)-1.
$ 
Since $|E(G)|=|E(H)|$
  we are done.
\end{proof}

\begin{defi}
\label{Ogdef}
Assume $b=0,1$. Set
 $$
 \Og:=\{(G, \oO_S):\  G\in \Sg,
 \  \oO_S\in \ocO^b(G-S)\}. 
 $$
Let   $(H, \oO_T), (G, \oO_S)\in \Og$. We  set  $(G, \oO_S)\leq (H, \oO_T)$ if 
$G\leq H$ and if
 there exists a contraction   $\gamma:G\to H$  such that
\begin{enumerate}
  \item
 \label{Ogdef1}
 $S \leq \gamma^*T$, or  equivalently (by \ref{fupr}\eqref{fuprb}), $ \gamma_*S \leq T$;
 \item
  \label{Ogdef2}
  $ \overline{\gamma}_*\oO_S \leq \oO_{T}$.
 \end{enumerate}
\end{defi}
The definition is illustrated in Figure 4.  
By   \eqref{Ogdef1} we have  
     $H-T\supset  H-\gamma_*S$. 
Hence $\oO_T$     can    be restricted to $H-\gamma_*S$.
By Definition~\ref{podef}, we require     that this restriction be equal  to $ \overline{\gamma}_*\oO_S$.

\begin{figure}[h]
\begin{equation*}
\xymatrix@=.2pc{
&&&&*{\bullet}  \ar@{-} @/^ 1pc/ [ddrr] \ar@{-}[ddrr]_e \ar@{-}@/_ 1pc/ [ddll] \ar@{-} [ddll]&&&& &&&& &&&&\\
 G =& &&&&&&& \ar@{->}[rrrr]^{\gamma} &&&&& H = & *{\bullet} \ar@{-} [rr] \ar@{-} @/_ 0.5pc/ [rr] \ar@{-} @/^ 0.5pc/ [rr] && *{\bullet} \ar@{-}@(ur,dr) \\ 
&&*{\bullet}\ar@{-}[rrrr] && &&*{\bullet}  && &&\\
&&\\
&&&&  \ar@{-} @/^ 1pc/ [ddrr]  \ar@{-}@/_ 1pc/ [ddll] &&&& &&&& &&&&\\
 S =& &&&&&&& \ar@{->}[rrrr]^{\gamma_*} &&&&& \gamma_* S = & \ar@{-} @/^ 0.5pc/ [rr] && \ar@{-}@(ur,dr) &&& \leq & T = & \ar@{-} @/^ 0.5pc/ [rr] &&\\ 
&&&& &&  && &&\\
&&\\
&&&&*{\bullet}  \ar@{->}[ddrr] \ar@{<-} [ddll]&&&& &&&& &&&&\\
 O_S =& &&&&&&& \ar@{->}[rrrr]^{\gamma_*} &&&&& \gamma_* O_S = & *{\bullet} \ar@{->} [rr] \ar@{<-} @/_ 0.5pc/ [rr] && *{\bullet} &&& \leq & O_T = & *{\bullet} \ar@{->} [rr] \ar@{<-} @/_ 0.5pc/ [rr] && *{\bullet} \ar@{->}@(ur,dr) &&  \\ 
&&*{\bullet}\ar@{<-}[rrrr] && &&*{\bullet}  && &&\\
&&}
\end{equation*}
\caption{An example of the partial order on $\overline{\mathcal{OP}}_{g}^0$: $(G, \oO_S) \leq (H, \oO_T)$ with $\gamma: G \rightarrow H$ contracting $e$. The orientations $O_S$, $\gamma_* O_S$ and $O_T$ are living on $G - S$, $H - \gamma_* S$ and $H - T$, respectively.}
\end{figure}
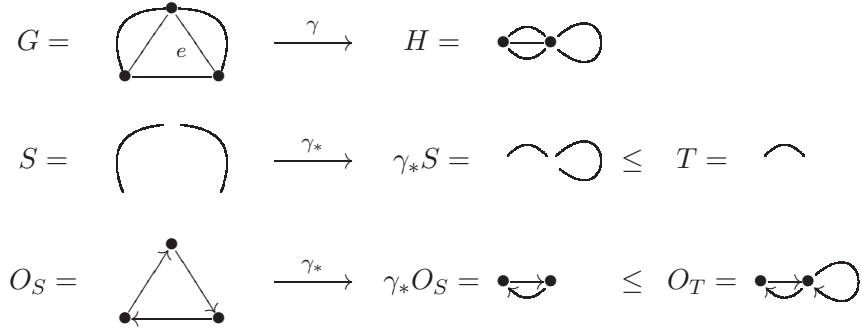

\begin{prop}
\label{propOg}
Assume $b=0,1$.
Then
 $\Og$ is a poset  such that   the inclusion $\OG \ha \Og$ is a morphism of posets for every $G\in \Sg$.
 Moreover, the following hold.
\begin{enumerate}[(a)]

\item
\label{propOgb}
The      forgetful  maps 
 $$
 \chi: \Og \la \Sg;\quad \quad (G, \oO_S)\mapsto G
   $$
   and     
   $$
\tau:  \Og \la \Ag  ;\quad \quad (G, \oO_S)\mapsto (G,S)
   $$
are  quotients of posets.
\item
\label{propOgc}
The following is a rank  on   $ \Og$
 $$
 \rho_{\Og}: \Og \la \N;\quad \quad (G,  \oO_S)\mapsto 3g-3-|E(G)|+g(G-S).
 $$

 \end{enumerate}
\end{prop}
\begin{proof}
 The only property of partial orders which is not an obvious consequence of the definition
 is    transitivity.
 Suppose $(G, O_S)\leq (H, O_T)$ and $(H, O_T)\leq (J,O_U)$, let
 $ \delta:H\to  J 
 $  be a   contraction. Then we have the following contraction,
 $ 
 \delta\circ \gamma:G\la  J. 
 $ 

Next, by \ref{forward}\eqref{forward2} we have
  $(\delta\circ \gamma)_*=  \delta_*\circ \gamma_*$.
  Hence, as  $\gamma_*S \leq T$  and  $\delta_*T \leq U$
 we have 
  $$
(\delta\circ \gamma)_*S= \delta_*(\gamma_*S) \leq \delta_*(T) 
 \leq  U
  $$
proving the first requirement of Definition~\ref{Ogdef}.
Finally, to show  that $\oO_{U}\geq (\ov{\delta\circ \gamma})_*\oO_S$
we must restrict $\oO_U$ to $J-(\ov{\delta\circ \gamma})_*S$ and check it is equal to $(\ov{\delta\circ \gamma})_*\oO_S$.
This is trivial.

\eqref{propOgb}. The map 
$ \chi:\Og \to \Sg$  factors as follows
$$
 \chi:\Og \stackrel{\tau}{\la} \Ag \la \Sg 
$$
and Proposition~\ref{Bgq} states that $\Ag \to \Sg$ is a quotient. Hence it suffices to prove that
$\tau$ is a quotient.  Now, $\tau$    is clearly a surjective morphism of posets. 
Let   $(G,S)\leq (H,T)$ and  let $\gamma:G\to H$ be a contraction 
such that $ S \leq \gamma^*T$. Now
pick $\oO_S\in \OG$,
then $\ov{\gamma}_*\oO_S\in \oOH$.
By Lemma~\ref{quoto}, there exists $\oO_T\in  \oOH$ such that  $\ov{\gamma}_*\oO_S\leq  \oO_T$.
As $\tau(G,\oO_S)=(G,S)$ and $\tau(H,\oO_T)=(H,T)$ the proof of \eqref{propOgb} is complete.

\eqref{propOgc}. Notice that $\rho_{\Og}(G,\oO_S)=\rho_{\Ag}(G,S)$, the latter being the rank   defined in Proposition~\ref{Bgq}.

Now, $\tau$ is such that if $\tau(G,\oO_S)=\tau(G',\oO_{S'})$ then $G=G'$ and $S=S'$, hence
$(G,\oO_S)$ and  $(G',\oO_{S'})$ are not comparable. Hence if     $(H,\oO_T)$  covers $(G,\oO_S)$
then $(H,T)$ covers $(G,S)$. Therefore   $\tau\circ \rho_{\Ag}=\rho_{\Og}$   is a rank on $\Og$.
 The proof is complete.
 \end{proof}

   \subsection{Automorphisms of graphs}
 \label{autsec}
 
We need to extend the functoriality results proved for edge-contractions in Section~\ref{secfun} to isomorphisms of graphs. 
We need the  following, whose proof is trivial.
\begin{prop}
\label{iso}
 Let $\alpha:G\to G'$ be an isomorphism.
 \begin{enumerate}
 \item
 \label{iso1}
Let $b=0,1$. Then we have an  isomorphism of posets
 $$\alpha_*:\AGb\la  \Ab_{G'}; \quad \quad  S \mapsto \alpha_*S=\alpha(S).$$
\item
\label{iso2}
  For   $O_S\in \cO^b(G-S)$ define  $\alpha_*O_S\in \cO^b(G'-\alpha_*S)$ so that, for any   $e\in E(G)$,  the starting half-edge of $\alpha(e)$  is the image under
  $\alpha$ of the starting half-edge of $e$.
Then   we have an isomorphism of   posets   $$
  \alpha_*:\OP^b_G{\la} \OP^b_{G'}; \quad \quad O_S\mapsto \alpha_*O_S.
  $$
\item
 \label{iso3}
The   isomorphism in \eqref{iso2} descends to an isomorphism of  posets 
    $$
 \alpha_*:\oOPGb {\la} \oOP^b_{G'}.
  $$
 \end{enumerate}
\end{prop}

\begin{definition}
We say that $(H, \oO_T), (G, \oO_S) \in \oOP_g^b$ 
are {\it conjugate}, and write 
 $(H, \oO_T) \equiv (G, \oO_S)$, if $G = H$ and there exists  $\alpha \in \Aut(G)$ such that $\alpha_*\oO_T  = \oO_S$. 
\end{definition}

Conjugacy  is clearly an equivalence relation on $\Og$. 
We denote
$$
\qOPGb=\oOPGb/_\equiv  \quad  \quad \text{and} \quad \quad  \qOg = \Og/_\equiv
$$
and write $ [O_S]$ and $(G, [O_S])$ for an element of $\qOPGb$ and   $\qOg$ respectively.

\begin{prop} \label{cOP} Notation as above.
We    endow $\qOg$ with the following partial order: 
$(G, [O_S]) \leq (H, [O_T]) $  if there exist $\oO_{T'} \in [O_T]$ and $\oO_{S'} \in [O_S]$ such that $(G, \oO_{S'}) \leq (H, \oO_{T'}) $ in $\oOP^b_g$. 

Then  the quotient   $\oOP^b_g \rightarrow [\mathcal{OP}^b_g]$ is a quotient of posets,  the inclusion $[\OPGb] \ha  [\mathcal{OP}^b_g]$   a morphism of posets, and the forgetful map $[\mathcal{OP}^b_g]\to \Sg$ is a quotient of posets. 
Furthermore $$\rho_{[\mathcal{OP}^b_g]}(G, [O_S]) = 3g - 3 - |E(G)| + g(G - S)$$ is a rank function. 

\end{prop}

\begin{proof}
Let $\gamma:G\to H$ be a contraction such that $(G,\oO_S)\leq (H,\oO_T)$.

By Lemma~\ref{posettriv}, it suffices to prove that for any $\oO_{S'}\equiv \oO_S$ there exists  $\oO_{T'}\equiv \oO_T$ such that $\oO_{S'}\leq \oO_{T'}$.
We have $\oO_S=\alpha_*\oO_{S'}$ for some $\alpha\in \Aut(G)$.  
If $\gamma$ is trivial then $\oO_S\leq \oO_T$ and $\oO_{S'}=\alpha_*^{-1}\oO_S\leq \alpha_*^{-1}\oO_T$, as $\alpha_*^{-1}\oO_T \equiv \oO_T$ we are done.

Suppose $\gamma$ nontrivial.
By hypothesis
$ (\oO_T)_{|H-\gamma_*S}= \ov{\gamma}_*\oO_S.
$ 
Let $\gamma'$ be the contraction obtained by composing $\alpha$ with $\gamma$:
$$
\gamma':G\stackrel{\alpha}{\la}G\stackrel{\gamma}{\la}H
$$
We have $\oO_{S'}\in \oO(G-\alpha^{-1}_*S)$; set $S'=\alpha^{-1}_*S$. We claim 
$$(\oO_T)_{|H-\gamma'_*S'}=\ov{\gamma}'_*\oO_{S'}$$  which   of course implies
$\oO_{S'}\leq \oO_T$.
We have
$$
\ov{\gamma}'_*\oO_{S'}=\ov{\gamma}_*\ov{\alpha}_*\oO_{S'}=\ov{\gamma}_* \oO_S=(\oO_T)_{|H-\gamma_*S}=(\oO_T)_{|H-\gamma_*\alpha_*\alpha^{-1}_*S}
 =(\oO_T)_{|H-\gamma'_* S'}  
$$
  as claimed.
Hence $[\mathcal{OP}^b_g]$ is a poset and  
 $\oOP^b_g \rightarrow [\mathcal{OP}^b_g]$ a quotient of posets.
 
 The inclusion  $[\OPGb] \ha [\mathcal{OP}^b_g]$ is obviously a morphism of poset.  

By Proposition~\ref{propOg}     the forgetful map $\chi:\oOP^b_g \to \Sg$ is a quotient of posets. It is clear that $\chi$ factors
as follows
 $$\chi:\oOP^b_g \la [\mathcal{OP}^b_g]\la  \Sg.
 $$ 
Since $\oOP^b_g \to [\mathcal{OP}^b_g]$ is a quotient,  $[\mathcal{OP}^b_g]\to \Sg$ is also a quotient.

The claim about the rank  follows from the fact that   conjugate elements of $\oOP^b_g$ have the same rank. 
\end{proof}

  \section{Stratifying the compactified universal   Picard variety.}
    \label{finsec} 
    In this section we turn to algebraic geometry and prove our main results.
    We work over an algebraically closed field $k$.
 \subsection{Dictionary between graphs and nodal  curves} 
  \label{dict} From now on, 
  $X$ will be an algebraic, projective, reduced curve over $k$    having at most nodes as singularities,  and whose   
   (weighted) dual graph  is    $G=(V,E)$.
Recall that $V$ is the set of irreducible components of $X$ and $E$ is the set
of nodes of $X$, with an edge/node joining the two vertices/components on which it lies.
The weight of a vertex/component is its geometric genus.
We shall use the same symbols for edges and nodes,
but we shall write $X=\cup_{v\in V}C_v$ with $C_v$   irreducible component.
 The   genus of $X$  equals  the genus of $G$, and $X$ is stable if so is $G$.
We shall   say, somewhat abusively,   that ``$X$ is dual to $G$".

 Let $S\subset E$ and let
 $ 
\nu_S:X^{\nu}_S\to X
$ 
the  normalization of $X$ at $S$. The dual graph of $\XN_S$ is $G-S$, and 
 $g(X^{\nu}_S)=g(G-S).$ 
  We denote by $ \hX_S$  the   nodal curve  
  obtained by attaching to $\XN_S$, for every node $e\in S$, a smooth rational component, 
  named {\it exceptional component}, to the two   branches of  $\nu_S^{-1}(e)$. 
   Of course, $X$ and $\hX_S$ have the same genus.
    
  If $X$ is a stable curve, the curves of the form $\hX_S$  are called {\it quasistable}.
Two exceptional components of a quasistable curve never intersect.

The dual graph of $\hX_S$  will be denoted by $\hG_S$. So,   $\hG_S$ is obtained from $G$ by inserting a  vertex of weight zero, $v_e$, in every edge $e\in S$. 
We refer to  $v_e$ as the {\it exceptional} vertex
corresponding to the exceptional component $C_{v_e}$ of $\hX_S$, and we write $h_e, j_e$ for the two edges of $\hG_S$ adjacent to $v_e$. We have
$
 \hX_S=\XN_S\cup (\cup_{e\in S}C_{v_e}).
$

The set of non-exceptional vertices of $\hG_S$ is naturally identified with $V(G)$. 
We denote $\hat{S}=\{h_e, j_e, \  \forall e\in S\}\subset E(\hG_S)$ so that we have a natural inclusion 
$ 
G-S\subset \hG_S-\hat{S}.
$ 

 Let $L$ be a line bundle on $X$, the multidegree of $L$ is defined as follows:
 $\mdeg (L)=\{\deg_{C_v}L,\  \  \forall v\in V\}$.
We  shall   identify $\mdeg (L)$ with a divisor on $G$, whose $v$-coordinate is $\deg_{C_v}L$, so that we have a map
$$
\mdeg:\Pic(X)\la \Div(G);\quad \quad L\mapsto \mdeg (L).
$$
Then
$ 
\Pic(X)=\sqcup_{\md \in \Div(G)}\Pic^{\md}(X)
$ 
where
 $\Pic^{\md}(X):=\mdeg^{-1}(\md)
$ 
is the moduli space of line bundles of multidegree $\md$. Of course, $\Pic^{\md}(X)$
 is isomorphic to the generalized Jacobian, $\Pic^{(0,\ldots,0)}(X)$, of $X$. 
    
  \subsection{Compactified   Jacobians   of a   curve}
  \label{CPsubsec}
Let
 $X$ be  a stable curve of genus $g$.
 We consider $\PXd$,  its  compactified degree-$d$  Jacobian.
 $\PXd$ is a connected, reduced, possibly reducible, projective variety of pure dimension $g$
 whose smooth locus  is a disjoint union
of (finitely many) $g$-dimensional varieties parametrizing line bundles of degree $d$ on $X$.
 
 Several constructions of $\PXd$  exist  in the literature,   \cite{OS}, \cite{Cthesis}, \cite{Simpson}, \cite{Esteves},
and, except for the one of \cite{Esteves}, they  have been proved to be isomorphic to one another
even though their modular interpretations
are different. 
We here adopt  the modular interpretation given in \cite{Cthesis},
according to which $\PXd$
 parametrizes
``stably balanced" line bundles  of degree $d$ on certain quasistable curves having   stable model $X$. To give the precise description we need some definitions.

\begin{defi}
\label{stabledef}
Let $G=\sqcup_{i=1}^cG_i$ have   $c$  connected components.
   \begin{enumerate}[(a)]
 \item
 \label{stableg}
A divisor $\md \in \Div^g(G)$ is {\it stable} if $c=1$  and if for every $Z\subset V(G)$
we have
$ 
  |\md_Z|>g(Z)-1.
 $ 
 
   \item
  Suppose $c=1$. A divisor $\md \in \Div^{g-1}(G)$ is {\it stable}  if   for every $Z\subsetneq V(G)$ we have
$ 
  |\md_Z|>g(Z)-1.
 $ 

\noindent
For arbitrary $c$,  a divisor $\md \in \Div^{g-c}(G)$ is stable if   its restriction
  to every   $G_i$ is stable of degree $g(G_i )-1$. 
  \end{enumerate}
  \end{defi}

 The  somewhat artificial  requirement, in    \eqref{stableg},  that    stable divisors of degree $g$ exist only on connected graphs, serves our goals and  simplifies  terminology.

As we are interested in the cases $d=g$ and $d=g-c$,
 we shall often unify our statements by writing
$$
d=g-c+b \quad \quad  \text{with} \quad \quad b=0,1.
$$
If $G$ is a graph of genus $g$ with $c$ connected components, for $b=0,1$
we set
$$
\Sigma^{b}(G):=\{\md\in \Div^{g-c+b}(G):\  \md \text{ is stable}\}.
$$

\begin{defi}
\label{stbaldef}
 Let $X$ be a stable curve of genus $g$ and $G$ its dual graph.   Let $S\subset E(G)$ 
and  $b=0,1$. A line bundle $\hL_S\in \Pic^{g-1+b}\hX_S$, and its multidegree 
$\mdeg\hL_S$,  are said to be {\it stably balanced} if
  
\begin{enumerate}[(a)]
 \item
$\hL_S$ has degree 1 on   each exceptional component;
  \item
    $\mdeg_{\XN_S} \hL_S$ is a stable  divisor  on $G-S$ of degree $g(G-S) -c(G-S)+b$.
    \end{enumerate}
  
 Line bundles  $\hL_S\in \Pic^{g-1+b}\hX_S$ as above are referred to as ``stably balanced line bundles of $X$".
 Two stably balanced line bundles, $\hL_S$ and $\hM_T$, of $X$ are {\it equivalent} if $S=T$ and if  
 their restrictions to $\XN_S$ are isomorphic.
\end{defi}

By definition, $\hL_S$ has total degree $g-1+b$ and
 degree $1$ on every exceptional component, hence   the restriction of $\hL_S$ to $\XN_S$ satisfies
$$
\deg_{\XN_S}\hL_S=g-1+b-|S|.
$$

 \begin{remark}
 \label{degst}
For $S\subset G$ we have
$$
\Sigma^{b}(G-S)=\{\md\in \Div^{g(G-S)-c(G-S)+b}(G):\  \md \text{ is stable}\},
$$
and   a  divisor in $\Sigma^{b}(G-S)$ has total degree $g-1+b-|S|$.
 \end{remark}
  From \cite{Cthesis}      we have
\begin{fact}
\label{Cthesis}
 Let $X$ be a stable curve of genus $g$ and let $b=0,1$. Then
 $ \PXb$ is a coarse moduli space for  equivalence classes of stably balanced line bundles
of degree $g-1+b$ of $X$.
\end{fact}
The above statement uses a different terminology   from the original one (\cite[Prop. 8.2]{Cthesis}) so we need a few words to explain that it is indeed the same.
If $b=0$ this is already known (see \cite{CV2} for example), so let us concentrate on the case $b=1$, i.e. degree $g$.
For degree $g$ the results of \cite{Cneron}, such as  Thm. 5.9, apply in their strongest form. 
Moreover, from Sect. 7 (in particular Lemma 7.6), we get that our definition~\ref{stbaldef} coincides with the definition of  stably balanced line bundles  given there.

 We need to establish an explicit   connection between  Definitions~\ref{stabledef} and \ref{stbaldef}.
For any   quasistable curve $\hX_S$  we 
  have a (not unique) contraction 
$$
\delta:\hG_S\la G=\hG_S/S_0,
$$
with $S_0=\{j_e, \  \forall e\in S\}$ where  $j_e$ is an edge of $\hG_S$ adjacent to the exceptional vertex $v_e$. Clearly, $\delta$  depends on the choice of
  $j_e$ for each $e\in S$.

Now,  let $\md\in \Div(G)$.
We denote by $\widehat{d}\in \Div(\hG_S)$ the following divisor

$$
\widehat{\md}_v: =\begin{cases}
{\md}_v & \text{ if } v\in V(G)\\
1 &  \text{ if }v=v_e,\  e\in S. 
\end{cases}
$$
In short,  $\widehat{d}$ extends $\md$
with  degree $1$ on all exceptional vertices.

We have the following  simple  fact, for which we use notation \eqref{mcv}.

\begin{lemma}
\label{exclm} 
Let $X$ be   stable   and $G$ its dual graph.
Let $\md_S$ be a stable divisor on $G-S$.
Then $\widehat{\md_S}$ is stably balanced and we have a surjective map
$$
\Pic^{\widehat{\md_S}}(\hX_S)\la \Pic^{\md_S}(\XN_S);\quad \quad \hL\mapsto \hL_{|\XN_S}.
$$
For any   $\delta:\hG_S\to G$ as above we have \  
$ 
\delta_*\widehat{\md_S} =\md_S+ \mc^{\delta}.
$ 
 
\end{lemma}
\begin{proof}
A divisor on $G-S$ is also a divisor on $G$, so the first part  follows trivially by definition.
Next, recall  that  $\mc^{\delta}_v$ is the number of edges   mapped  to $v\in V(G)$ by $\delta$.
Hence 
$ 
 \mc^{\delta}_v = 
0$ if $\delta^{-1}(v)=v$, and   $
 \mc^{\delta}_v =1$  otherwise. 
Since the value of $\widehat{\md_S}$ on exceptional vertices is $1$ we have
$$
(\delta_*\widehat{\md_S})_v =\begin{cases}
(\widehat{\md_S})_v & \text{ if } \delta^{-1}(v)=v\\
(\widehat{\md_S})_v+1 &  \text{  otherwise.}
\end{cases}
$$
Hence
$  
(\delta_*\widehat{\md_S} - \mc^{\delta})_v = (\widehat{\md_S})_v =(\md_S)_v.
$  
\end{proof}

  \subsection{Combinatorics of   compactified  Jacobians}
  We shall now connect to the material of the earlier sections.
  
 \begin{lemma} \label{orient}

Let $G$ be connected of genus $g$. 

\begin{enumerate}[(a)]

\item
\label{orienta} Let $\underline{d} \in \Div^{g-1}(G)$. There exists a $0$-orientation, $O$  such that  $\underline{d} = \underline{d}^O$ if and only if $|\underline{d}_{Z}| \geq g(Z) - 1$ for all  $Z\subset V$.

\item 
\label{orientb} For any $\underline{d} \in \Sigma^1(G)$   there exists a $1$-orientation, $O$, on $G$ such that $\underline{d} = \underline{d}^O$.

\end{enumerate}

\end{lemma}

\begin{proof}

Part \eqref{orienta} is well known, for example in graph theory as a version of Hakimi's Theorem (for a modern formulation see \cite [Thm 4.8] {ABKS}).

For part \eqref{orientb}, fix a vertex $v$ of $G$. Let $\md':=\md-\underline{v}$ so that $\md' \in \Div^{g-1}(G)$. We   have  $|\underline{d}'_{Z}| \geq g(Z) - 1$ for all   $Z\subset V$. Indeed, if $v \in Z$, we get $|\underline{d}_Z'| = |\underline{d}_Z| - 1 > g(Z) - 2$;  thus   $|\underline{d}_Z'|  \geq g(Z) - 1$. If $v \not \in Z$ we get $|\underline{d}_Z'| = |\underline{d}_Z| > g(Z) - 1$. Thus, by part \eqref{orienta}, we can choose a $0$-orientation $O'$ on $G$ such that  $\underline{d}' = \underline{d}^{O'}$.  Since $\underline{d} \in \Sigma^1(G)$, we have $|\underline{d}_{G  - v}| > g(G - v) - 1$, hence$$
\underline{d}_v =g-  |\md _{G  - v}| < g - g(G - v) + 1\leq g(v) +  \deg v -1 + 1 =g(v) +  \deg v
$$ 
(the ``$\leq$" above is a ``$=$" iff $G-v$ is connected).
On the other hand
$$
\md_v=\md'_v+1=g(v)+\mt^{O'}_v.
$$
Therefore $\mt^{O'}_v< \deg v$, hence $O'$ has an edge, 
  $e$, whose   source  is  $v$. Biorienting $e$ gives a $1$-orientation, $O$, with $\underline{d} = \underline{d}^O$. 
\end{proof}
  Recall that we denote by  $\cO^0(G)$ (resp. $\cO^1(G)$) the set of 
  totally cyclic    (resp.  rooted)   orientations
  on $G$, and by $ \OPt_G$  (resp. $ \OPb_G$) the poset of 
   totally cyclic    (resp.  rooted)   orientations
 on spanning subgraphs of $G$.
On   such sets we defined an equivalence relation 
whose   class-sets are marked     by an overline.
Finally, recall the notation  introduced in \ref{degst}.
\begin{lemma}
\label{degor}
 Let $b=0,1$. Let $G$ be a graph of genus $g$ and $S,T\subset E$.
Consider  the following map
 $$
 \oOPGb\la \Div(G);\quad \quad \overline{O}_S\mapsto \md^{O_S}.
 $$
 
\begin{enumerate}[(a)]
 \item
 \label{degora}
The  map
 induces a bijection between $ \ocO^b(G-S)$ and  $\Sigma^{b}(G-S)$.
\item
 \label{degorb}
If $O_S$ is a $b$-orientation   with
 $\md^{O_S}\in \Sigma^{b}(G-S)$, then $O_S\in \cO^b(G-S)$.
 \end{enumerate}
 \end{lemma}
\begin{proof}
    The map is well defined and injective  by Definition~\ref{defequivo}. Its image lies in $ \Sigma^{b}(G-S)$ by  
   Lemma~\ref{lm0} in case $b=0$  and by Lemma~\ref{lmO1} in case $b=1$. Moreover, its image is   the whole of $ \Sigma^{b}(G-S)$ by
Lemma~\ref{orient}. 
This proves  \eqref{degora},  and  \eqref{degorb}  follows from it.
\end{proof}
 
\begin{remark}
\label{pointsCP}
 By \ref{Cthesis} the points of $\PXb$ correspond to equivalence classes of stably    balanced 
line bundles, and
  two such line bundles are equivalent if  
  they are defined on the same $\hX_S$ and  if their restrictions to $\XN_S$ are isomorphic.
 Denote by $\md_S$ a stable divisor of $G-S$ and by  $P_X^{\md_S}\subset \PXb$
the set of equivalence classes of stably balanced line bundles on $\hX_S$ whose restriction to  $\XN_S$
has multidegree $\md_S$.
 By Lemma~\ref{degor}, there exists a unique   
 ${\oO_S}\in \ocO^b(G-S)$ such that $\md_S= \md^{O_S}$, and every stable divisor on $G-S$ is obtained in this way.
 Therefore we   define, for any  ${\oO_S}\in \cO^b(G-S)$ 
\begin{equation}
 \label{POS}
P_X^{O_S}:=P_X^{\md_S}. 
\end{equation}
\end{remark}
 \begin{thm}  
\label{PX}
Let $X$ be a stable curve of genus $g$ and $G$ its dual graph, let $b=0,1$. Then the following is a graded stratification of $\PXb$    by   $\oOPGb$\begin{equation}
 \label{strato}
\PXb=\bigsqcup_{{\oO_S}\in  \oOG^b(G)}P_X^{O_S},
\end{equation}
and we have     natural isomorphisms for every  ${\oO_S}\in  \oOG^b_G$:
\begin{equation}
 \label{isoo}P_X^{O_S}\cong \Pic^{\md^{O_S}}(X^{\nu}_S).
 \end{equation}
 \end{thm}
\begin{proof}
The case $b=0$ follows   from results of  \cite{CV2}. As
our proof in case $b=1$  also works for $b=0$,   we include it for completeness.

As   in Remark~\ref{pointsCP},
we denote by $P_X^{\md_S}\subset \PXb$
the set of equivalence classes of stably balanced line bundles on $\hX_S$ whose restriction to  $\XN_S$
has degree $\md_S$,
for $\md_S$ a stable divisor of $G-S$.
By    Fact~\ref{Cthesis} we have

\begin{equation}
 \label{stratd}
\PXb=\bigsqcup_{\stackrel{S\subset E}{\md_S\in  \Sigma^{b}(G-S)}}P_X^{\md_S}.
\end{equation}
Now, as noted above, we have 
 $ 
 P_X^{O_S}=P_X^{\md_S}
 $ 
for a unique class 
 ${\oO_S}\in \cO^b(G-S)$ such that $\md_S= \md^{O_S}$. Moreover, by Lemma~\ref{degor} every   $\md_S\in  \Sigma^{b}(G-S)$ is obtained in this way, for every $S\subset E$. Hence  \eqref{stratd} yields \eqref{strato}.
 
Also, we obviously have 
 $ 
P_X^{\md_S} \cong \Pic^{\md_S}(\XN_S)$, 
 from which \eqref{isoo} follows.  
 
Next, recalling Definition~\ref{stratdef}, we prove the following
$$
{P_X^{O_S}}\subset  \ov{P_X^{O_T}} \quad \Leftrightarrow \quad {\oO_S}\leq{\oO_T}.
$$ 
 By  \cite[Prop. 5.1]{Cthesis}
 (revised using graphs)
 we have  ${P_X^{\md_S}}\cap  \ov{P_X^{\md_T}}\neq \emptyset$ if and only if 
${P_X^{\md_S}}\subset  \ov{P_X^{\md_T}}$.
Moreover, ${P_X^{\md_S}}\subset  \ov{P_X^{\md_T}}$ 
 if and only if
$T\subset S$ and the edges in $S\ssm T$ can be oriented so that, denoting by
$t_v$ the number of edges  in $S\ssm T$ with target a vertex $v$, we have
 $(\md_{T})_v= (\md_{S})_v+t_v. $ 
 
Assume ${P_X^{\md_S}}\subset  \ov{P_X^{\md_T}}$ and denote by $O'_T$ the orientation on $G-T$ which extends $O_S$ to $S\ssm T$ by the orientation we   just defined (where $\oO_S\in \ocO(G-S)$ is such that $\md^{O_S}=\md_{S}$, by the previous part).
Of course  $O_S\leq O'_T$ and, as     $\md^{O_T}=\md_{T}$   for some $\oO_T \in \ocO(G-T)$,  we   have
 $\md^{O_T}=\md_{T}=\md^{O'_T},$ 
hence $O'_T\sim O_T$. We conclude that ${\oO_S}\leq{\oO_T}$. The converse is obvious.

Finally,  we need to show that the stratification \eqref{stratd} is graded. Recall that
the generalized Jacobian of $\XN_S$ is an irreducible variety of dimension $g(G-S)$, hence so is $ \Pic^{\md_S}(\XN_S)$, hence so is $P_X^{O_S}$. By Proposition~\ref{poo}, the map
${\oO_S}\mapsto g(G-S)$  is a rank  $\oOPGb$, hence we are done.
 \end{proof}

\subsection{Specialization of polarized curves}
\label{rigor}
We shall be interested in  (flat, projective) families of curves over a one-dimensional nonsingular base, specializing to a given curve $X$.
Up to restricting the base    we can assume that 
away from $X$ the family is topologically trivial, i.e. that every fiber different from  $X$ has   the same dual graph  of some fixed   curve $Y$.
We shall   refer to  such a family  as  a {\it specialization from $Y$ to $X$}.
Since $X$ has only nodes as singularities, the same holds for $Y$. 
 Suppose our curves   $X$ and $Y$   are ``polarized", i.e.    endowed with  a line bundle, 
$L\in \Pic(X)$ and  $M\in \Pic(Y)$.
We say that  $(Y,M)$ specializes  to $(X,L)$ 
if there is a specialization of $Y$ to $X$  under which $M$ specializes to $L$. Let us be more precise.

The family under which $Y$ specializes to $X$ is a projective morphism
$f:\cX\to B$ where $B$ is a smooth, connected, one-dimensional variety with a   point $b_0$ such that $f^{-1}(b_0)\cong X$, 
and   the restriction of $f$ away from $b_0$ is locally trivial, moreover   $f^{-1}(b)\cong Y$  for some $b\neq b_0$.
As an \'etale base change of $f$ determines again a specialization of $Y$ to $X$ 
we are free to replace $f$ by an \'etale base change.
For the polarized version, to say that $M$ specializes to $L$ means   that $\cX$
is endowed with a line bundle whose restriction to $Y$ is $M$ and whose restriction to $X$ is $L$.

\begin{prop}
\label{hbk1}
 Let $X$ and $Y$ be two nodal curves and   $G$ and $H$   their respective dual graphs.
 Let $L\in \Pic(X)$ and  $M\in \Pic(Y)$ such that
  $(Y, M)$ specializes to $(X,L)$. Then    there exists a contraction
$\gamma : G\to H$ 
 such that
$$
\gamma_*\mdeg (L)= \mdeg(M).
$$
 \end{prop}
  
In the opposite direction, we have the following.
\begin{prop}
\label{hbk2}
Let  $\gamma : G\to H$
be  a contraction between two graphs.
 Then for any curve $X$ dual to $G$  and for any $L\in \Pic(X)$
 there exist a curve $Y$ dual to $H$  and a line bundle $M\in \Pic(Y)$
 such that $\gamma_*\mdeg (L)= \mdeg(M)$ and such that
 $(Y, M)$ specializes to $(X,L)$.
\end{prop}

\begin{proof}
We prove Propositions~\ref{hbk1} and \ref{hbk2} together as their proofs are closely related.
They extend
 \cite[Thm 4.7   (2)]{CHBK}
  to polarized nodal  curves.
    
To prove Proposition ~\ref{hbk1},    assume   $(Y, M)$ specializes to $(X,L)$. Under such a specialization every node of $Y$ specializes to a node of $X$
  and different nodes specialize to different nodes. Hence we   partition $E(G)=S_0\sqcup T$ so that
  $S_0$ is the set of nodes of $X$ which are not specializations of nodes of $Y$.
  We let $\gamma:G\to G/S_0$, and, arguing  as for  \cite[Thm 4.7]{CHBK}, we have   $G/S_0=H$.

For any vertex $w\in V(H)$ we write $D_w\subset Y$ for the irreducible component corresponding to $w$.
As shown in loc. cit., the specialization from $Y$ to $X$ induces a specialization  of 
 $ 
D_w$ to $ \cup_{ \gamma(v)=w}C_v$  (as a   subcurve of $X$).
Now, $M$ specializes to $L$ and hence $M_{|D_w}$ specializes to the restriction of $L$ to
$\cup_{ \gamma(v)=w}C_v$.
Therefore
$$
\mdeg(M)_w=\deg_{D_w} M=\deg L_{|\cup_{ \gamma(v)=w}C_v}=\sum_{\gamma(v)=w} \mdeg(L)_v =\gamma_*\mdeg(L)_w.
$$
  This proves
 Proposition~\ref{hbk1}.
 
 For Proposition~\ref{hbk2},
  let $\gamma:G\to G/S_0=H$ be a contraction, for some $S_0\in E(G)$; write $E(G)=S_0\sqcup T$ 
so that $T$ is identified with $E(H)$.
 Let $X$ be a curve dual to $G$ and let $\XN_T$ be its normalization at   $T$, so that $G-T$ is the dual graph
 of $\XN_T$. The curve $\XN_T$ is endowed with $|T|$ pairs of marked smooth points, namely the points lying over the nodes in $T$.
 Observe that the connected components of  $\XN_T$ are in bijection with the connected components of $H-T$,
 and hence with the vertices of $H$. We can therefore decompose   $\XN_T$ as follows
 $ 
 \XN_T=\sqcup_{w\in V(H)}Z_w
 $  
with $Z_w$  a connected nodal curve whose   genus, $g(Z_w)$, is equal to the weight of $w$ as a vertex in $H$.
  Therefore we can find a family of smooth    curves of genus $g(Z_w)$ specializing to $Z_w$,
  i.e. we have a smooth curve, $W_w$, specializing to $Z_w$. Considering the union for $w\in V(H)$ we get a specialization
  of $\sqcup_{w\in V(H)}W_w$ to $\sqcup_{w\in V(H)}Z_w=  \XN_T$.

Now, up to \'etale cover,  such a specialization
 can be  endowed with
 $|T|$
 pairs of sections specializing to the above $|T|$ pairs of   marked points of $\XN_T$.
 By gluing together each such pair of sections we get  a specialization   to our $X$ from a curve, $Y$, whose dual graph
 is $H$.
 
Clearly,    the contraction $\gamma:G\to H$ corresponds to this specialization from $Y$ to $X$.
  
 Now, using the notation of    Subsection~\ref{rigor}, let $f:\cX\to B$ be a family under which $Y$ specializes to $X$,
 and consider its relative Picard scheme, $\Pic_{\cX/B}\to B$. Its fiber over $b_0$ is $\Pic(X)$ and its fiber over $b$ is $\Pic(Y)$.
 Write $\md=\deg{L}$; 
  we claim that, in the relative Picard scheme, $\Pic^{ \md}(X)$ is the specialization of $\Pic^{\gamma_*\md}(Y)$.
  Indeed,   $\Pic^{ \md}(X)$ must be the specialization of some   connected component of $\Pic(Y)$
    (even if  this Picard scheme were not separated,   every connected component of its fiber over $b_0$ is the specialization of some connected component of the general fiber),
  and this component is necessarily
  $\Pic^{\gamma_*\md}(Y)$ by
  the same computation we used to prove Proposition~\ref{hbk1}.
  
Now, as $\Pic^{ \md}(X)$ 
 is the specialization of $\Pic^{\gamma_*\md}(Y)$, any $L\in \Pic^{ \md}(X)$ is   the specialization of  some $M\in  \Pic^{\gamma_*\md}(Y)$, and we are done.
 \end{proof}

 \subsection{Compactified universal Jacobians}
 \label{unisec}
We fix $d\in \Z$.
In this paper we are    interested in   $d=g$ and $d=g-1$, so we shall restrict to these two cases even though some of  the preliminary results quoted in  this subsection hold more generally for every $d$. We also assume $b=0,1$ so that $d=g-1+b$.

We let  $\Mgb$ be the moduli space of stable curves of genus $g\geq 2$, an   irreducible  projective variety. 
 \begin{fact}\label{strathbk}
The following is a graded stratification of $\Mgb$   by $\Sg$:
$$
\Mgb= \bigsqcup_{G\in \Sg} M_G 
$$
 where $M_G$  parametrises   curves having $G$ as dual graph.
 \end{fact}
Indeed,  the map $G\mapsto \dim M_G$ is the  rank on $\Sg$   defined in Proposition~\ref{rkSg}.
Now, from  \cite{Cthesis} we introduce, for every $d\in \Z$,   the
 compactified universal degree-$d$ Jacobian 
$$
\psi_{g,d}:\Pdgb \la \Mgb.
$$
We sometimes write $\psi=\psi_{g,d}$ for simplicity.
Recall   that $\Pdgb$ is the GIT quotient of a Hilbert scheme, and that $\psi$ is a projective morphism whose fiber over $X\in \Mgb$
is isomorphic to $\PXd/\Aut(X)$.  
Set
$$ 
P_G^{d}:=\psi_{g,d}^{-1}(M_G).
$$
 Pick a stable curve $X\in M_G$.
Then we have a   canonical  map
 \begin{equation}
\label{muX}
 \mu_X:\PXd\to P_G^d.
   \end{equation}

\begin{cor}
\label{corstr}
 Let $G,H\in \Sg$. Then
 $$
 P_G^{d} \subset \ov{ P_H^{d}}\quad \quad \text{if and only if} \quad \quad H\geq G.
 $$
\end{cor}
\begin{proof}
 It suffices to use  Fact~\ref{strathbk}  and  that $\psi:\Pdgb \to \Mgb$ is   projective.
\end{proof} 
In the next remark we recall the basic moduli properties  of $\Pdgb$.
  \begin{remark}
\label{modprop}

Let $f:\cX\to B$ be a family of quasistable curves of genus $g$ and let $\cL$ be a line bundle on $\cX$ whose restriction, $L_b$,
to every fiber over  $b\in B$  is   stably balanced  of degree $d$
  (in the sense of Definition~\ref{stbaldef}). Then there is a {\it moduli morphism},
$\mu_{\cL}:B \to \Pdgb$ such that the image of $b\in B$ is the equivalence class of $L_b$.

Consider the   case of a fixed curve  rather than a family. So   $B=\{b\}$   and $\cX=\hX_S$ is a fixed quasistable curve. Let $L, L'\in \Pic(\hX_S)$
be stably balanced. If the restriction of $L$ and $L'$ away from the exceptional components 
 are isomorphic (i.e. if  $L_{\XN_S}\cong L'_{\XN_S}$)  then $\mu_L(b)=\mu_{L'}(b)$.
 \end{remark}
Fix $G$ and $S\subset E(G)$.
Let $f:\cX\to B$ be a family of stable curves all having  dual graph identified with $G$, hence 
  $S$ can be identified with   $|S|$ (set-theoretic) sections    of $f$ corresponding to the nodes in $S$.
Denote by $f_S:\cX_S\to B$ the desingularization at these   sections, so that every fiber of $f_S$ has dual graph $G-S$.
For every $e\in S$ we have a pair of sections  $(\sigma_{h^+_e}, \sigma_{h^-_e})$ of $f_S$
(where $h^+_e, h^-_e$ are the half-edges of $e$). We glue to $\cX_S$ 
a copy of $\PP^1\times B$  by identifying its $0$ and   $\infty$ section to  $ \sigma_{h^+_e}$ and  $\sigma_{h^-_e}$.
By repeating  this for every $e\in S$
we obtain a
family of quasistable curves  $\widehat{f_S}:\hat{\cX}_S\to B$ with dual graph $\hG_S$.

Let now $\md_S$ be a   divisor on $G-S$,  denote by  $\Pic^{\md_S}_{{f_S}}$ the corresponding connected component
of the Picard scheme $\Pic_{f_S}$. Similarly, denote by $ 
\Pic^{\widehat{\md_S}}_{\widehat{f_S}} 
$ 
the connected component of $\Pic_{\widehat{f_S}}$ corresponding to  $\widehat{\md_S}\in \Div(\hG_S)$.

Now, using the notation in Lemma~\ref{exclm}, we have
\begin{lemma}
\label{exclm2} Let $f:\cX\to B$ be as above.
Let $b=0,1$ and 
  $\md_S\in \Sigma^b(G-S)$.
Then there exist a moduli morphism
$ 
\mu_{\widehat{\md_S}}:\Pic^{\widehat{\md_S}}_{\widehat{f_S}}\to \Pdgb
$ and a morphism $\mu_{\md_S}: \Pic^{\md_S}_{{f_S}}
\to P_G^d 
$
such that
$$
\mu_{\widehat{\md_S}}:\Pic^{\widehat{\md_S}}_{\widehat{f_S}} \stackrel{\varphi}{\la}  \Pic^{\md_S}_{{f_S}}
 \stackrel{\mu_{\md_S}}
{\la} P_G^d,
$$
where $\varphi$ is given by restriction away from the exceptional components.
\end{lemma}

\begin{proof}
We have a   polarized family of quasistable curves
$$
\cL\la \Pic^{\widehat{\md_S}}_{\widehat{f_S}}\times_B \hat{\cX}_S\la \Pic^{\widehat{\md_S}}_{\widehat{f_S}} $$
where
$\cL$ is 
the  tautological  (Poincar\'e)  line bundle,  which, by hypothesis, is relatively stably balanced. 
By Remark~\ref{modprop} we have a 
  moduli morphism $\mu_{\cL}:\Pic^{\widehat{\md_S}}_{\widehat{f_S}}\to \Pdgb$. Set     $\mu_{\widehat{\md_S}}=\mu_{\cL}$, 
it is clear that its image   lies in $P_G^d.$

We let  $\varphi:\Pic^{\widehat{\md_S}}_{\widehat{f_S}} \to  \Pic^{\md_S}_{{f_S}}$ be the map given by restricting a line bundle away from the exceptional components,  so $\varphi$  is the analog of the map used in Lemma~\ref{exclm}.  Now,  as we said in Remark~\ref{modprop}, if two line bundles have the same image under $\varphi$
(i.e. their restriction away from the exceptional components are isomorphic)
they also have the same image under $\mu_{\cL}$. 
By applying a standard argument using that $\Pdgb$ is a GIT quotient, we conclude that  there exists a map
$\mu_{\md_S}:\Pic^{\md_S}_{{f_S}}
\to  P_G^d$ such that
$\mu_{\widehat{\md_S}}$ factors as stated.
\end{proof}

	\subsection{The strata of $\Pdgb$}
	Our   goal is to find a    stratification of $\Pdgb$  compatible with the one of $\Mgb$.
By \cite[Prop. 3.4.1]{ACP}, the stratum $M_G$ has the following presentation
\begin{equation}
 \label{pi}
 \widetilde{M}_G :=\Pi_{v\in V}M_{w(v),\deg(v)} \stackrel{\pi}{\la} M_G= \widetilde{M}_G /\Aut(G)
\end{equation}
where    $M_{w(v),\deg(v)}$ is the moduli space of smooth curves  of genus $w(v)$ with $\deg(v)$ marked points
representing the branches/half-edges over the nodes/edges. 

More generally, with the notation of  \cite[Subsection 2.1]{ACP}, for every $S\subset E$, consider the $2|S|$-marked graph 
{\bf G-S}, whose underlying (unmarked) graph is $G-S$, and whose  $2|S|$-marking is given by the half-edges corresponding to $S$.
Then {\bf G-S} is stable as marked graph and we have a moduli space, $M_{\text{\bf{G-S}}}$, of  stable curves with $2|S|$ marked points and dual graph {\bf G-S}.
 In particular,  if $S=E$ then $ \widetilde{M}_G =M_{\text{\bf{G-E}}}$
 and
the map $\pi$ above factors:
$$
 \pi: \widetilde{M}_G\stackrel{\pi_S}{\la} M_{\text{\bf{G-S}}}=\widetilde{M}_G /\Aut(\text{\bf{G-S}}) \la M_G.
$$

For our goal  we   need  a  universal  curve over $ \widetilde{M}_G$, but it is well known that this   may fail to exist over some   $M_{w(v), \deg(v)}$.
However (see \cite{GAC} for example),  
a universal  curve   exists over some finite cover of it.
We   choose a finite cover $M'_{w(v), \deg(v)}\to M_{w(v), \deg(v)}$ of large enough degree (the same for all pairs $w(v), \deg(v)$) so that we 
have a universal curve over each $M'_{w(v), \deg(v)}$.
We let   $\widetilde{M}_G'$ be the product of the $M'_{w(v), \deg(v)}$ for $v\in V$ so that, composing with \eqref{pi}, we have a finite map
$ 
\pi':\widetilde{M}_G' \to M_G.
$  The action of $\Aut(G)$ on $M_G$ lifts naturally to an action on $\widetilde{M}_G'$.

We denote by  $ \cC_{w(v),\deg(v)}\to M'_{w(v), \deg(v)}$ the  universal curve, and
we     have  the following   family $$
\widetilde{\cX} _G: =\sqcup_{v\in V} \cC_{w(v),\deg(v)}\la \widetilde{M}_G',
$$ together with $2|E|$  sections, $\sigma_h: \widetilde{M}_G' \to \widetilde{\cX}_G$, indexed by the half-edges of $G$.

Fix $S\subset E$.  Let    $\widetilde{\cX} _G\to   \cX_G^S $ 
be the gluing   along   pairs  $(\sigma_{h^+_e}, \sigma_{h^-_e})$ 
	for every  $e\not\in S$. Then $\cX_G^S $ is a family over the space
  $Z_G^S:=\widetilde{M}'_G /\Aut(\text{\bf{G-S}})$.
Let  $$
f_S: \cX_G^S \to Z_G^S 
$$ 
be this family of curves, all of whose fibers 
have dual graph  $G-S$.
Since $ Z_G^S$ is  a finite cover of $M_{\text{\bf{G-S}}}$,  the   map $\pi'$  factors through finite maps:   
	$$ \pi': \widetilde{M}'_G \la Z_G^S\la M_G.  $$
Fixing a stable multidegree $\underline{d}^{O_S}$ on $G - S$, by  Lemma~\ref{exclm2}  
 we   get a morphism  
 $$\mu_{O_S}:\Pic_{f_S}^{\md^{O_S}} \la P_G^d.$$ 
We define $P_{G}^{O_S}$ to be the image of this map.

 \begin{lemma} 
 \label{PGqp}
 Let $G\in \Sg$ and  $O_S\in \OPGb$ with $b=0,1$. Then
 $P^{O_S}_G$  is   quasiprojective, irreducible  of dimension $3g-3-|E(G)|+g(G-S)$.
 
  If $O_T\equiv O_S$ for some $O_T\in \OPGb$, then $P^{O_S}_G=P^{O_T}_G$.
 
 \end{lemma}
 \begin{proof}
 The    morphism $\mu_{O_S} $  is finite because so is the morphism $Z_G^S\to M_G$. Moreover 
 $\mu_{O_S} $
 exhibits $P_G^{O_S}$ as the image  
 of an irreducible quasiprojective variety of dimension
$$
\dim\Pic^{\md^{O_S}}_{f_S}= \dim Z_G^S+g(G-S)=3g-3-|E(G)|+g(G-S) 
$$
(as $\dim Z_G^S=\dim M_G$). So  the first part of the statement is proved.

Now suppose $O_T\equiv O_S$, then   $O_T  = \alpha_* O_S$ for some $\alpha \in \Aut(G)$.
Hence $\alpha_*\md^{O_S}=\md^{O_{T}}$,  and $\alpha$ induces an isomorphism between   $Z_G^S$ and $ Z_G^{T}$, 
a corresponding isomorphism between   $\cX_G^S$ and $\cX_G^T$, and an isomorphism 
$\Pic_{f_S}^{\md^{O_S}}\cong \Pic_{f_T}^{\md^{O_T}}$. The latter    induces an isomorphism between the respective Poincar\'e line bundles.
   Therefore the images of  $\mu_{O_S}$ and $\mu_{O_T}$ in $P^d_G$ get identified; see the second part of Remark~\ref{modprop}.
\end{proof}
    We    define   for any $[O_S]\in \qOPGb$ 
$$
P_G^{[O_S]} := P_G^{O_S}, 
$$
  by     Lemma~\ref{PGqp}, this   does not depend on   the representative in $[O_S]$.

\subsection{Stratifications of universal Jacobians in degree  $g-1$ and  $g$.}

\begin{thm} \label{thmg-1}
 The following is a graded stratification of $\Pbb$ by  \  $\qOg$
$$
\Pbb =\bigsqcup_{(G,[O_S])\in \qOg} P_G^{[O_S]}.$$

\end{thm}

\begin{proof}
We have
$$
 \Pbb = \bigsqcup_{G\in \Sg} \Biggl(\  \bigsqcup_{[O_S] \in  \qOPGb}P_G^{[O_S]}\Biggr)= 
\bigsqcup_{(G,[O_S])\in \qOgb} P_G^{[O_S]}.$$

Indeed, the only thing that might not be clear is that the union   is disjoint. 
Suppose two different strata $P_G^{[O_S]}$ and $P_G^{[O_T]}$ intersect and let us show they coincide.
Let $X\in \Mgb$ be such that $P_G^{[O_S]}\cap P_G^{[O_T]}\cap \psi^{-1}(X)$ is not empty.
Recall  that the   strata $P^{O_S}_X$ and $P^{O_T}_X$ are disjoint in $\PXb$.   Since automorphisms of $X$ obviously map strata to strata in $\PXb$, the images via $\mu_X$ (see \eqref{muX})  of $P^{O_S}_X$ and $P^{O_T}_X$ are no longer disjoint if and only if there is an automorphism $\alpha_X$ of $X$ identifying them. Then one easily checks that the induced automorphism $\alpha$ on $G$ maps $\oO_S$ to $\oO_T$ in $\OG$. Hence   $ [O_S] = [O_T]$.

 Lemma~\ref{PGqp}  gives that $P_G^{[O_S]}$ is quasiprojective, irreducible,       of dimension
 $$
\dim P_G^{[O_S]}=  \dim P_G^{O_S} = 3g-3-|E(G)|+g(G-S),
$$
and  by  Proposition~\ref{cOP}   the right hand side is a rank  on $\qOgb$. 

To complete the proof we must show  that 
we have   a stratification in the sense of Definition~\ref{stratdef}.
 We will do that in the next two propositions.
 \begin{prop} 
\label{prop1}
Let  $(G,\oO_S),(H,\oO_T)\in \Og$.
If 
 $(G,\oO_S)\leq (H,\oO_T)$   then $P_G^{[O_S]}\subset \ov{P_H^{[O_T]}} $.
\end{prop}

\begin{proof}
Consider $\psi: \ov{P}_g^{g - 1 + b}\rightarrow \Mgb$. For a fixed $X \in M_G$ we have $$P_G^{[O_S]} \cap \psi^{-1}([X]) = \bigcup_{\oO_{S'} \equiv \oO_S} \mu_X(P_X^{O_{S'}}),$$ 
with $\mu_X$  defined in \eqref{muX}. It suffices to show that for every such $X$   and every $\oO_{S'} \equiv \oO_S$,   every point in 
$ 
P_X^{O_{S'}}$ is a specialization of line bundles parametrized by ${P_H^{O_T}}$, so that that $\mu_X(P_X^{O_{S'}})\subset \ov {P_H^{O_T}}$.

By the proof of Proposition~\ref{cOP} we have that for any $\oO_{S'} \equiv \oO_S$ there is $\oO_{T'} \equiv \oO_T$ with $\oO_{S'} \leq \oO_{T'}$. Since $P_G^{[O_S]} = P_G^{[O_{S'}]}$ and $P_G^{[O_T]} = P_G^{[O_{T'}]}$ we can replace $\oO_{S'}$ by $\oO_S$ without loss of generality.

By hypothesis,   there exists a curve   $Y$ dual to $H$    which specializes to $X$; 
 let  $\gamma:G\to H$ be the associated contraction.
 Under the corresponding 
 specialization of compactified Picard varieties, $\PYb$ specializes to $\PXb$.

Now,  $\ov{\gamma}_*{\ov{O}}_S\in \oOPHb$, hence $\md^{ \gamma_*O_S}$ is stable,
 and hence, by \ref{exclm}, $P_Y^{ \gamma_*O_S}$ parametrizes stably balanced line bundles on
 $\hY_R$ of degree $\widehat{\md^{\gamma_*O_S}}$, where $R=\gamma_*S$. 
 We begin by showing that  $P_Y^{ \gamma_*O_S}$ specializes to $P_X^{O_S}$.
To the contraction $\gamma$ we   associate the contraction
 $$
 \hgamma: \hG_S\la \hH_R=\hG_S/\hat{S_0} 
 $$
 (where $\hat{S_0}=\delta_E^{-1}(S_0)$ for $\delta:\hG_S\to G$).
Now,
with the notation introduced before Lemma~\ref{exclm}, consider $\widehat{\md^{O_S}}$ and   $\widehat{\md^{\gamma_*O_S}}$. We claim  
\begin{equation}
 \label{hat}
 \widehat{\md^{\gamma_*O_S}}=\hgamma_*\widehat{\md^{O_S}}.
\end{equation}

 Let $v\in V(\hH_R)$. If $v=v_e$ for $e\in R$ then 
 $v_e$ is also an exceptional vertex of $\hG_S$ mapped to $v_e$ by $\hgamma$.
 Hence both divisors appearing in \eqref{hat} have value $1$ on $v_e$.
Now suppose $v\in V(H)$, then, by Proposition~\ref{fprop},
 $$
 (\widehat{\md^{\gamma_*O_S}})_v=(\md^{\gamma_*O_S})_v=(\gamma_*\md^{O_S})_v+\mc^{\gamma,S}_v
 =\sum_{z\in \gamma_V^{-1}(v)}\md^{O_S}_z+\mc^{\gamma,S}_v=(\hgamma_*\widehat{\md^{O_S}})_v
 $$
 where the last equality follows as    $\mc^{\gamma,S}_v$ is equal to  the number of exceptional vertices of $\hG_S$ that are mapped to $v$ by $\hgamma$.
 \eqref{hat} is proved.
 
We can now apply  Proposition~\ref{hbk2},  to obtain that any line bundle $\hL\in \Pic(\hX_S)$ such that $\mdeg\hL=\widehat{\md^{O_S}}
$ is obtained as specialization of a line bundle $\hM\in \Pic(\hY_R)$ such that
$$
\mdeg \hM=\hgamma_*\mdeg\hL=\hgamma_*\widehat{\md^{O_S}}= \widehat{\md^{\gamma_*O_S}}.
$$
  This proves that $P_Y^{ \gamma_*O_S}$ specializes to $P_X^{O_S}$. Therefore
$ 
 \mu_X(P_X^{O_S})\subset \ov{P_H^{ \gamma_*O_S}}.
$ 
Now, by Theorem~\ref{PX} and the hypothesis $\overline{\gamma}_*\oO_S\leq{\oO_T}$,  in $\PYb$ we have 
$ 
P_Y^{ \gamma_*O_S}\subset \ov{P_Y^{O_T}}.
$ 
Hence
$ 
\mu_X(P_X^{O_S})\subset \ov{P_H^{ \gamma_*O_S}}\subset \ov{P_H^{O_T}}.
$ 
\end{proof}
\begin{prop}
\label{prop3}
Let  $(G,[O_S])$ and $ (H,[O_T])$ be in $[\mathcal{OP}_g^b]$.
The following are equivalent
\begin{enumerate}[(a)]
\item
\label{prop3a}
  $ P_G^{[O_S]}\cap \ov{P_H^{[O_T]}}\neq \emptyset$.
\item 
\label{prop3b}
 $(G,[O_S])\leq (H,[O_T])$.
    \item
   \label{prop3c}
    $ P_G^{[O_S]}\subset \ov{P_H^{[O_T]}}$.
   \end{enumerate}
\end{prop}

\begin{proof}
 \eqref{prop3a} $\Rightarrow$\eqref{prop3b}.
 By hypothesis, we have
 a specialization
 of polarized curves, $
(\hY_T, \hM )$ to $(\hX_S, \hL )$, 
where $X$ and $Y$ are curves dual to $G$ and $H$ respectively, and  $ \hL $ and  $\hM $ are
stably balanced 
 line bundles on  $\hX_S$ and $\hY_T$ such that
  $\mdeg_{\XN_S}  \hL =\md^{O_{S'}}$ and 
   $\mdeg_{ Y^{\nu}_T}  \hM =\md^{O_{T'}}$ for some  $\oO_{S'} \in [O_S]$ and $\oO_{T'} \in [O_T]$. 
 It suffices to prove that     $\oO_{S'}\leq\oO_{T'}$.
 
To simplify the notation, from now on we      drop the indices  and write  $\oO_{S'}=\oO_S$ and  $\oO_{T'}=\oO_T$.
 We denote by  $\hG_S$ and   $\hH_T$ 
 the dual graphs of $\hX_S$ and  $\hY_T$.  
By  Proposition~\ref{hbk1},  the above specialization is associated to  a   contraction
 $$
 \hgamma: \hG_S \la \hH_T,
 $$
 such that
 $ \hgamma_*\mdeg \hL =\mdeg \hM $.
Now,  every exceptional component of $\hY_T$ specializes to an exceptional   component  of
$\hX_S$,   hence  we have a specialization of $Y$ to $X$ and the associated contraction
 $\gamma:G\to H=G/S_0.$  We have an inclusion  $T\subset S$ induced by $E(H)\subset E(G)$. 
 
Denote by $\hat{O}$ the orientation on $\hG_S$ obtained from  $O_S$ by orienting all edges adjacent to exceptional vertices towards the exceptional vertex. Then the degree of $\un{d}^{\hat{O}}$ on each exceptional component is $1$ and $\un{d}^{\hat{O}} = \widehat{(\un{d}^{O_S})}$. 
 
We first assume $T=\emptyset$,
then $ \hH_T=H$  and we have a commutative diagram 
$$
\xymatrix@=.4pc{
  &&&& &&&& \hG_S  \ar[rrrd]_{\delta}    \ar[rrrrrr]^{\hgamma} &&&&&&H  &&&   &&&&  \\
  &&&& &&&&&&&G \ar [rrru]_{\gamma}   &&& &&  
}
$$

Here $\delta$ is given as follows: every exceptional vertex $v_e$ in $\hG_S$ has two adjacent edges $h_e$ and $j_e$, both $\hat{O}$-oriented towards $v_e$. Defining $\delta$ amounts to choosing one of the two for every exceptional vertex. If $e \in S_0$, we can contract any of the two, as $\hgamma$ contracts both. If $e \not \in S_0$ we choose the one contracted by $\hgamma$. This choice clearly makes the diagram commutative. 
Set  $O' :=  \delta_* \hat{O} $       on $G$. Since $ \hgamma_*\mdeg \hL =\mdeg \hM $, i.e. $\hgamma_*(\un{d}^{\hat{O}}) = \un{d}^{O_{\emptyset}}$  we get 
$$\gamma_* O' = \gamma_* (\delta_* \hat{O}) = \hgamma_* (\hat{O}) \sim O_{\emptyset}$$ 
where the equivalence at the end follows from
  Proposition \ref{fprop} (\ref{dlm}),
(with  $\mc^{\hgamma, \emptyset}=0$ because  $O'$ is defined on the whole graph).
  By construction we have $O'_{|G - S} = O_S$, i.e. $O_S \leq O'$, and thus by Proposition \ref{fprop} (\ref{fprop4}): 
  $$\gamma_* O_S \leq \gamma_* O' \sim  O_{\emptyset},$$ 
  which proves the claim in case $T = \emptyset$. 

In general, we have $T\subset S$ and, of course, $T\cap S_0=\emptyset$.
Therefore  the restriction of  $\gamma$ to $G-T$ is
$$
\gamma_{|G-T}:G-T \la \frac{G-T}{S_0}=H-T.
$$
Write $G'=G-T$, $H'=H-T$ and $\gamma'=\gamma_{|G-T}$.
Then write $O'=(O_T)_{|H'}$ and $O'_{S'}=(O_S)_{|G'}$  with $S'=S\ssm T$.
By  the previous case  
 $\overline{\gamma'_*O'_{S'}}\leq \overline{O'}$,
 i.e.
\begin{equation}
 \label{gammaprime}
 \gamma'_*O'_{S'}\sim O'_{|H'-\gamma'_*S'}.
\end{equation}
Now, $O_S$ is defined on $G-S\subset G-T$, hence
 $$
  \gamma'_*O'_{S'}=(\gamma_{|G-T})_*(O_S)_{|G-T}=\gamma_*O_S.
 $$
 Also, as $O_T$ is defined on $H'=H-T$, we have
 $$
 O'_{|H'-\gamma'_*S'}=((O_T)_{|H-T})_{|H-T-\gamma'_*S'}=(O_T)_{H-\gamma_*S}
 $$
($\gamma'_*S'\cup T=S'\ssm S_0 '\cup T=S\ssm S_0 =\gamma_*S$ as $T\cap S_0=\emptyset$).
 Combining with
 \eqref{gammaprime} gives 
 $\gamma_*O_S\sim (O_T)_{H-\gamma_*S}
$
 and we  are done with the implication  \eqref{prop3a} $\Rightarrow$\eqref{prop3b}.

  \eqref{prop3b} $\Rightarrow$\eqref{prop3c}.
  By hypothesis,  $\oO_{S'}\leq \oO_{T'}$ for some $\oO_{S'}\in [\oO_S]$ and $\oO_{T'}\in [\oO_T]$.
 By Proposition~\ref{prop1},   
we have $ P_G^{[O_{S'}]}\subset \ov{P_H^{[O_{T'}]}}$, hence we conclude as follows
  $$
P_G^{[O_S]}= P_G^{[O_{S'}]}\subset \ov{P_H^{[O_{T'}]}}= \ov{P_H^{[O_T]}}.
 $$
  
      \eqref{prop3c} $\Rightarrow$\eqref{prop3a} is obvious. 
\end{proof}
 Theorem~\ref{thmg-1}  is   proved, and so is Theorem~\ref{mainintro}.
\end{proof}  

\noindent
{\it Acknowledgements.} We are grateful to Spencer Backman, Roberto Fringuelli and Margarida Melo for useful conversations related to this paper, and to the referee for some helpful   comments.

\bibliography{strataCC}
\

\noindent MSC (2010): 14H10; 14H40; 05Cxx 
\end{document}